\documentclass[oneside,11pt]{article}
\usepackage{amsmath}
\usepackage{amssymb}
\usepackage{bbm}
\usepackage[margin=1.5cm]{geometry}
\usepackage{hyperref}
\usepackage{tikz}
\usetikzlibrary{decorations.pathmorphing,arrows}
\usepackage{braket}

\newtheorem{thm}{Theorem}[section]

\newtheorem{lemma}[thm]{Lemma}

\newtheorem{theorem}[thm]{Theorem}
\newtheorem{remark}[thm]{Remark}
\newtheorem{proposition}[thm]{Proposition}
\newtheorem{definition}[thm]{Definition}

\newenvironment{proof}{{\bf Proof:}}{\hfill$\square$\vskip.5cm}
\newenvironment{proofof}{}{\hfill$\square$\vskip.5cm}
\newcommand{\R}{\mathbb{R}}
\newcommand{\N}{\mathbb{N}}
\newcommand{\C}{\mathbb{C}}

\newcommand{\Z}{\mathbb{Z}

}

\renewcommand{\Z}{{\mathbb{Z}}}

\title{Generating Function for Pinsky's Combinatorial Second Moment Formula for the Generalized Ulam Problem}

\date{January 19, 2023}

\author{
Samen Hossein\\
\small The Bronx High School of Science\\[-2pt]
\small 75 Bronx Science Bvd\\[-2pt]
\small Bronx, NY 10468\\[7pt]
Shannon Starr\footnote{ Partially funded by Simons Collaboration Grant}\\
\small Department of Mathematics\\[-2pt]
\small University of Alabama at Birmingham\\[-2pt]
\small University Hall, Room 4005\\[-2pt]
\small 1402 Tenth Avenue South\\[-2pt]
\small Birmingham, AL 35294-1241\\[-2pt]
\small \href{mailto:slstarr@uab.edu}{slstarr@uab.edu}}

\begin{document}

\maketitle

\abstract{Given a uniform random
permutation $\pi \in S_n$, let $Z_{n,k}$ be equal to the number of increasing subsequences of length $k$:
so
$Z_{n,k}=|\{(i_1,\dots,i_k) \in \Z^k\, :\, 1\leq i_1<\dots<i_k\leq n\, ,\ \pi_{i_1}<\dots<\pi_{i_k}\}|$.
In an important paper, Ross Pinsky proved $\mathbf{E}\big[Z_{n,k}^2\big]$
is equal to $\sum_{i} A(k-i,i)B(n,2k-i)$, where for any nonnegative integers $N$ and $j$, we have $B(N,j) =  \binom{N}{j}/j!$ and $A(N,j)$ is a particular nonnegative
integer, which Pinsky characterized in two different ways. One characterization of $A(N,j)$
involves the occupation time 
of the $x$-axis prior to a first return to the origin.
Using this, he proved a law of large numbers for the sequence $Z_{n,k_n}$ when $k_n=o(n^{2/5})$
as $n \to \infty$.
In a follow-up paper, he also proved the sequence $Z_{n,k_n}$ fails to obey a law of large numbers
when $1/k_n = o(1/n^{4/9})$ as $n \to \infty$.
Here, we return to his combinatorial formula for the the second moment of $Z_{n,k}$,
and we obtain a generating function for the $A(N,j)$ triangular array.
We are motivated by the hope of applying spin glass techniques to the well-known Ulam's problem
to see if this gives a new perspective.
}

\thispagestyle{empty}

\section{Introduction}

\subsection{Extended abstract}

In two papers Ross Pinsky considered a random variable $Z_{n,k}$ equal to the number of increasing subsequences of length $k$
in a uniform random permutation $\pi \in S_n$. 
That investigation may be called the ``generalized Ulam problem,'' because: 
if $L_n$ is the length of the longest increasing subsequence for a uniform random $\pi \in S_n$, then $\left(L_n< k\right)
\Leftrightarrow \left(Z_{n,k}=0\right)$. 
The generalized Ulam problem is to investigate the distribution $\mathbf{P}(Z_{n,k}\leq j)$,
possibly asymptotically as $n \to \infty$, for both variables $j$ and $k$.
%The original Ulam problem was just for $\mathbf{P}(Z_{n,k}=0)$ or alternatively $1-\mathbf{P}(Z_{n,k}=0)=\mathbf{P}(L_n\geq k)$.

%Pinsky found various interesting results. 
%We consider Pinsky's combinatorial formula
%for the 2nd moment. 
Among other things, Pinsky proved $\mathbf{E}\big[Z_{n,k}^2\big]$
is equal to $\sum_{i} A(k-i,i)B(n,2k-i)$, where for any nonnegative integers $N$ and $j$, we have $B(N,j) =  C(N,j)/j!$.
Let us use $C(N,j)$ for the binomial
coefficient $(j! (N-j)!)^{-1} N!$, throughout.
The combinatorial array $A(N,j)$, which is a nonnegative
integer for each $N,j \in \{0,1,\dots\}$, is
equal to $4^{N} \mathbf{E}[Q_{N,j} \cdot R_N]$,
where $R_N$ and $Q_{N,j}$ are random variables defined on the sample space for a $d=2$ dimensional SRW.
If a sample path of the SRW up to time $2N$ is $(0,0)=(U_0,V_0),(U_1,V_1),\dots,(U_{2N},V_{2N})$, then 
$Q_{N,j}$ is the number of choices of $(t_0,\dots,t_{j-1}) \in \Z^{j}$
such that $0=t_0\leq t_1\leq \dots\leq t_{j-1}<2N$ satisfying $U_{t_1}=\dots=U_{t_j}=0$, and 
$R_N$
is the indicator for the event that $(U_{2N},V_{2N})=(0,0)$.
If the occupation time of the $y$-axis is $\tau_{N}$ then $Q_{N,j}$ equals $C(\tau_{N}+j-1,j)$
by the ``stars-and-stripes'' problem.%related to the occupation of the $y$-axis. %which is related to the occupation time of the $y$-axis.

Pinsky obtained rigorous general bounds which are sufficient to determine the applicability of the Paley-Zygmund method to prove a weak law of large numbers for $Z_{n,k}$: $k=k_n$ is {\fontsize{6}{7}{${\mathcal{O}}$}}$(n^{2/5})$ as $n \to \infty$.
In a follow-up paper he also proved a regime where the weak law does not hold for $Z_{n,k}$:
$k=k_n$ is such that $1/k_n$ is  {\fontsize{6}{7}{${\mathcal{O}}$}}$(1/n^{4/9})$ as $n \to \infty$.
We calculate the generating function for his $A(N,j)$ sequence. Our method is 
the complex analysis approach, as in the excellent reference of Pemantle and Wilson.
We find a closed formula for $\alpha(z,w)=\sum_{N=0}^{\infty} \sum_{j=0}^{2N} A(N,j) z^N w^j$ 
which involves elliptic integrals of the third kind $\Pi$.
Recall P\'olya's
generating function formula for the return time of a random walk: $g(z) = \sum_{N} z^{2N} \mathbf{E}[R_N]$ equals $K(z)/\pi$, where $K$ is the complete elliptic integral of the first kind.
This satisfies $g(z)=\alpha(z,0)$.
We also give some motivation.

\subsection{Set-up}

%\fontsize{10}{11}
{
Ulam's problem is a problem that involves a combination of tools and techniques
from the intersection of various subfields and specialities in mathematics, statistics  and physics.
It has been studied by many top researchers.
The main results are beautiful but diffficult.

The depth of the results related to Ulam's problem is well-known.
There are good reviews. 
See, for example, the review article of Aldous and Diaconis \cite{AldousDiaconis}. 
Many new and important results have been discovered since the publication of that particular paper.
One notable recent development is the directed landscape perspective \cite{Virag}.

We will state the actual problem known as Ulam's problem.
Let $\pi$ be a random permutation in $S_n$, chosen according to the uniform measure.
So for any permutation in $S_n$, the probability for $\pi$ to equal that one permutation is 
$1/n!$.
Ulam's problem involves increasing subsequences for $\pi$.
Given a cardinality $k$ subset of $\{1,\dots,n\}$, we may write the set as 
$\{i_1,\dots,i_k\}$ for 
\begin{equation}
1\leq i_1<\dots<i_k\leq n\, . 
\end{equation}
Then we say that 
$(i_1,\dots,i_k)$ is an increasing subsequence for $\pi$ if 
\begin{equation}
\pi_{i_1}\, <\, \dots\, <\, \pi_{i_k}\, .
\end{equation}
If we fix $(i_1,\dots,i_k)$ and vary over $\pi$, then the probability of the event,
``$(i_1,\dots,i_k)$
is an increasing subsequence for $\pi$,'' is equal to $1/k!$.
This is because if we are given $k$ numbers $\{j_1,\dots,j_k\}$
satisfying
\begin{equation}
\label{eq:jIncreasing}
1\, \leq\, j_1\, <\, \dots\, <\, j_k\, \leq\,  n\, ,
\end{equation}
and we are told that we have equality of the sets
\begin{equation}
\label{eq:jpiSubset}
	\{\pi_{i_1},\dots,\pi_{i_k}\}\, =\, \{j_1,\dots,j_k\}\, ,
\end{equation}
then there are still $k!$ possible permutations $\sigma=(\sigma_1,\dots,\sigma_k) \in S_k$ 
such that 
\begin{equation}
\label{eq:jpiOrdered}
	\forall r \in \{1,\dots,k\}\, ,\ \text{ we have }\ \pi_{i_r} = j_{\sigma_r}\, .
\end{equation}
But for $(i_1,\dots,i_r)$ to be an increasing subsequence for $\pi$, then we want $\pi_{i_r} = j_r$ for all $r$.
That means $\sigma = (1,2,\dots,k)$, the identity element of $S_k$, which is just one of the $k!$ possible permutations.

Ulam's problem is to determine the distribution of $L(n)$, the length of the longest increasing subsequence for $\pi$:
the largest $k$ such that there is at least one subsequence $(i_1,\dots,i_k)$ which is increasing for $\pi$.
Hammersley used Markov's bound and Kingman's subadditive theorem to garner information about $L(n)$
in the large $n$ limit.
Let $Z_{n,k}(\pi)$ denote the number of subsequences of cardinality $k$ which are increasing for $\pi$:
\begin{equation}
\label{eq:defZnk}
	Z_{n,k}(\pi)\, =\, \sum_{i_1,\dots,i_k=1}^{n} \prod_{r=1}^{n-1} \left(\mathbf{1}_{(0,\infty)}(i_{r+1}-i_r) \cdot \mathbf{1}_{(0,\infty)}(\pi_{i_{r+1}}-\pi_{i_r})\right)\, .
\end{equation}
Then $\sum_{i_1,\dots,i_k=1}^{n}  \prod_{r=1}^{n-1} \mathbf{1}_{(0,\infty)}(i_{r+1}-i_r)$ is equal to $n$-choose-$k$,
and for each $(i_1,\dots,i_k)$ we know 
\begin{equation}
	 \mathbf{P}(\pi_{i_1}<\dots<\pi_{i_k})\, =\, \frac{1}{k!}\, ,
\end{equation}
where the probability is over the uniform measure on $\pi \in S_n$ (with $(i_1,\dots,i_k)$ held fixed).
 Therefore,
\begin{equation}
\label{eq:Zmean}
	\mathbf{E}[Z_{n,k}]\, =\, \binom{n}{k} \cdot \frac{1}{k!}\, .
\end{equation}
Then, using Markov's inequality, it is easily seen that
\begin{equation}
\label{ineq:Markov}
	\mathbf{P}(Z_{n,k}>0)\, =\, \mathbf{P}(L(n)\geq k)\, \leq\, \mathbf{E}[Z_{n,k}]\, .
\end{equation}
Then, by Stirling's formula, writing $x$ for $k/\sqrt{n}$,
\begin{equation}
\label{eq:FirstMoment}
	\binom{n}{k} \cdot \frac{1}{k!}\,  \sim\, \frac{\exp\left(-2x n^{1/2}\, \ln\left(x/e\right) -\frac{1}{2}\, x^2+ \delta_n(x)\right)}{2 \pi x n^{1/2} (1 - x n^{-1/2})^{1/2}}\, ,
\end{equation}
in the limit $k \to \infty$, $n-k \to \infty$, where
\begin{equation}
	\delta_n(x)\, =\, x n^{1/2} \ln\left(1-x n^{-1/2}\right)- 
n\, \ln\left(e^{x n^{-1/2}}\left(1-x n^{-1/2}\right)\right)+\frac{x^2}{2}\, ,
\end{equation}
which would be $o(1)$ as long as $xn^{-1/2} = o(1)$.
In \cite{Hammersley}, Hammersley used the subadditive ergodic theorem to prove that 
\begin{equation}
	\liminf_{n \to \infty} n^{-1/2}\, \mathbf{E}[L_n]\, =\, \limsup_{n \to \infty} n^{-1/2}\, \mathbf{E}[L_n]\, =\, c\, ,
\end{equation}
for some number $c$ which {\em a priori} could have been in $(0,\infty]$.
But by equation (\ref{eq:FirstMoment}), he noted $c\leq e$.
In fact, an equally simple argument using the Erd\"os-Szekeres theorem
gives $c\geq 1/2$.
See, for example, Section 1.4 of J.M.~Steele's monograph \cite{Steele}.

The actual mean is approximately $c\, n^{1/2}$ for $c=2$,
by a result of Vershik and Kerov \cite{VershikKerov}, and a partial result of Logan and Shepp \cite{LoganShepp}.
Moreover, there is a
complete characterization of the leading-order fluctuations for large $n$, according to the Baik-Deift-Johannson theorem
\cite{BDJ}.
And new results are being continually discovered for the leading-order fluctuations, such as the ``directed landscape'' picture \cite{Virag}.

\subsection{Higher moments for the generalized Ulam's problem}
Suppose one seeks information for the distribution of $Z_{n,k}$. The event that $Z_{n,k}>0$
is equal to the event that $L_n\geq k$.
Thus, finding the distribution of $Z_{n,k}$ generalizes Ulam's problem, so it can be called the
``generalized Ulam problem.''
The easy bounds above for the event $Z_{n,k}>0$ follows from Markov's bounds.
But more precise bounds could in principle be obtained if we knew higher moments of $Z_{n,k}$.

In principle, the Bonferroni inequalities would give sharp bounds if one had sufficient information
on the moments. A good reference for this  topic is the combinatorics textbook
of Charalambides \cite{Charalambides}.
In Appendix \ref{sec:Bon} we will review this.
The first author, Samen Hossein, has an earlier unpublished article, submitted to a private organization, which treats this topic in more detail. That is available upon 
request. But, we will summarize it in the appendix.
Because of all this, it is interesting to know the distribution of $Z_{n,k}$, beyond just its mean.

In \cite{Pinsky1}, Pinsky investigated the second moment.
He obtained an exact combinatorial formula.
He then rigorously analyzed the formula obtaining upper and lower bounds of matching exponential growth
rates  although with different constant
multipliers.
\begin{proposition}[Pinsky's First Result]
If there is a sequence $k_n$ such that $\lim_{n \to \infty} k_n /n^{2/5}=0$ then 
\begin{equation}
	\frac{\mathbf{E}[Z_{n,k_n}^2]}{\big(\mathbf{E}[Z_{n,k_n}]\big)^2} - 1\, =\, \mathcal{O}\left(\frac{k_n^{5/2}}{n}\right)\, .
\end{equation}
If instead $\liminf_{n \to \infty} k_n/n^{2/5} > 0$ and $\limsup_{n\to\infty} k_n/n^{2/5} < \infty$ then 
\begin{equation}
	0\, <\, \liminf_{n \to \infty} \frac{\mathbf{E}[Z_{n,k_n}^2]}{\big(\mathbf{E}[Z_{n,k_n}]\big)^2}\, \leq\, 
	\limsup_{n \to \infty} \frac{\mathbf{E}[Z_{n,k_n}^2]}{\big(\mathbf{E}[Z_{n,k_n}]\big)^2}\, < \infty\, .
\end{equation}
If $\lim_{n \to \infty} n^{2/5}/k_n = 0$ then 
\begin{equation}
	\frac{\big(\mathbf{E}[Z_{n,k_n}]\big)^2}{\mathbf{E}[Z_{n,k_n}^2]}\, =\, \mathcal{O}\left(\frac{n}{k_n^{5/2}}\right)\, .
\end{equation}
\end{proposition}
Using this, he could determine exactly the criteria for the Paley-Zygmund second moment 
method to guarantee that $Z_{n,k}$ satisfies a weak law of large numbers type result.
Namely, that condition is that $\lim_{n \to \infty} k_n/n^{2/5}=0$.

Pinsky noted that just because the Paley-Zygmund method does not imply a weak law, that does
not mean that a weak law may not still be true, since it does not necessarily require vanishing of the variance
of the rescaled quantity.
So in \cite{Pinsky2}, Pinsky proved a threshold to be able to disprove a weak law for $Z_{n,k}$.
He showed
\begin{equation}
\forall \epsilon>0\, ,\ \lim_{n \to \infty} \frac{1}{n!} 
\left|\left\{\pi \in S_n\, :\, \left|\frac{Z_{n,k_n}(\pi)}{\mathbf{E}[Z_{n,k_n}]}
-1\right|<\epsilon\right\}\right|\, =\, 1\qquad \Rightarrow\qquad
\liminf_{n \to \infty} \frac{n^{4/9}}{k_n}\, >\, 0\, .
\end{equation}
He also  proved interesting equivalent formulations of his questions, and he has 
conjectures which remain as fascinating open problems.

We may follow the paradigm of iterating a random optimization problem.
See for example the {\em objective method} of Aldous and Steele \cite{AldousSteele}.
Especially, see the section on the ``Stalking horse,'' and distributional identities.
For the higher moments of the generalized Ulam problem,
the stalking horse is the covariance or even just the mixed moments
$$
\mathcal{Q}_n(k,\ell)\, =\, \mathbf{E}[Z_{n,k} Z_{n,\ell}]\, .
$$
This seems to be the first step towards iterating to go from the second moment to higher moments.
For example, for the third moment $\mathbf{E}[Z_{n,k} (Z_{n,k})^2]$, we may want to start
with $Q_n(k,2k)$. More precisely, the product $A(k-i,i) \cdot B(n,2k-i)$ gives the expected
number of pairs of increasing subsequence of $\pi \in S_n$, where each of the two subsequences
is of length $k$, and the number of points in the intersection is $i$. As before, $B(N,j)= C(N,j)/j!$
where $C(N,j)$ is the usual binomial coefficient $N! / (j! (N-j)!)$, throughout this article.

We would want to calculate the expectation for the pair of increasing subsequences of $\pi \in S_n$
where the first subsequence has length $k$ and the second subsequence has length $\ell$,
and the intersection has cardinality $i$.
Even though this seems like a clear goal, we will not pursue that goal in this article.
We hope to carry it out in a future article, either together or perhaps in a different collaboration
with one or both of the current authors.
Or perhaps, another researcher will do it first.
We hope that this is okay: we believe that the results already contained in the present article
are of interest to some readers, even without that further work (which is not done yet).
In this article, we just study the 2nd moment.

\subsection{Motivation for exact solutions}

In theoretical physics and in the mathematical physics of spin glasses, one is interested in the
precise form of the second moment and all higher moments, at least to leading order asymptotically
as $n \to \infty$.
Even though 
$\mathbf{E}[Z_{n,k}^2]/(\mathbf{E}[Z_{n,k}])^2$ may be too large to apply the 
Paley-Zygmund second moment method,
one may be interested in $\ln\left(\mathbf{E}[Z_{n,k}^2]\right)/\Big(2\ln\left(\mathbf{E}[Z_{n,k}]\right)\Big)$.
For example, in spin-glass theory, physicists moved from formulas for the moments
to predictions about the actual distribution.
In Appendix \ref{sec:SK}, we review this briefly.

An excellent reference for the limitations of classical mathematical tools when trying to apply
them to spin glass techniques was written by van Hemmen and Palmer \cite{vanHemmenPalmer}.
However, the spin glass techniques have proved effective in illuminating the Sherrington-Kirkpatrick
mean field spin glass problem, and may be useful beyond that.
See for example, the classical physics textbook \cite{MPV}.
The first rigorous proof of Parisi's 
ansatz for the SK model, given by Talagrand, seemed to proceed by very different techniques \cite{Talagrand},
with a starting point the beautiful approach of Guerra \cite{Guerra} and Guerra and Toninelli \cite{GuerraToninelli}.

In Appendix \ref{sec:SK} we give a rapid review of some points of that development.
Our purpose in the appendix is to demonstrate the role of calculating higher moments in the physicists' ``replica 
symmetry breaking'' approach.
In the present paper, we are interested in the higher moments of $Z_{n,k}$, starting with the second moment.
We call this the generalized Ulam problem.
Technically, the generalized Ulam problem should be the problem of stating as much as possible
about the distribution of $Z_{n,k}$ for the doubly-indexed sequence of numbers for $n$ and $k$.
But we are most hopeful to be able to calculate the positive integer moments, in some asymptotic sense.
See Section \ref{sec:Outlook} for a discussion of the relation between the 2nd moment and higher moments.

Another motivation involves Pinsky's second characterization of $A(N,k)$ in terms of the occupation
time of a $d=2$ dimensional simple random walk, where the set for the occupation time is the $y$-axis.
%For Brownian motion, the fraction of time spent in the $y$-axis is $0$ because $d=1$ dimensional
%Brownian motion spends a fraction of time $0$ at the origin.
We discuss this more right after Remark \ref{rem:PinskySRW}.
P\'olya calculated the generating function for returns to the origin in the $d=2$ dimensional
simple random walk, in terms of the complete elliptic integral of the first kind.
The main result of our paper is related to that, in so much as it is a direct generalization.
But our formula is more complicated, so that it involves complete elliptic integrals of the third kind (which is 
the most general setting for complete elliptic integrals).

\subsection{Outline for the remainder of the paper}

The moments of the generalized Ulam problem appear to have been investigated primarily by Ross Pinsky.
We begin, in the next section, by stating Pinsky's precise combinatorial formula for the second moment.
Then we calculate a generating function for the quantities that appear in Pinsky's decomposition,
with hopes of returning later and extracting the precise asymptotics from the generating function.
For us, the calculation of the generating function was already difficult.

In Section 2, we present Pinsky's combinatorial second moment formula, and we set-up the moment generating function.
In Section 3, we present the main tools for calculating the moment generating function.
These are Cauchy's integral formula and the diagonal method, which allows one to specialize multivariate generating
functions by integrating over certain contours to restrict the terms in the double sum (or multiple sums) occuring in the series.
In Section 4 we give formulas for the generating function, with several different equivalent formulas being presented.
A number of technical lemmas and side-issues are presented in various appendices.
%\subsection{Formula for generating function and relation to other problems}

}

\section{Pinsky's combinatorial  second moment formula}

{
In his paper \cite{Pinsky1}, Pinsky calculated
\begin{equation}
	\mathbf{E}\left[Z_{n,k}^2\right]\, =\, \sum_{i=0}^{k}A(k-i,i)\, B(n,2k-i),
\end{equation}
where for any $N,j \in \{0,1,\dots\}$, we have
\begin{equation}
	B(N,j)\, =\, \binom{N}{j} \cdot \frac{1}{j!}\, ,
\end{equation}
and 
\begin{equation}
	A(N,j)\, =\, \sum_{\substack{\ell_0,\dots,\ell_j \in \{0,1,\dots\}\\ \ell_0+\dots+\ell_j=N}}\
	\sum_{\substack{m_0,\dots,m_j \in \{0,1,\dots\}\\ m_0+\dots+m_j=N}}\
	\prod_{r=0}^{j} \left(\binom{\ell_r+m_r}{\ell_r,m_r}\right)^2\, .
\end{equation}
See Pinsky's Proposition 2. We will not reproduce his proof, but we mention that his argument is elegant.

Let us define the new quantity
\begin{equation}
	K(L,M,j)\, =\, \sum_{\substack{\ell_0,\dots,\ell_j \in \{0,1,\dots\}\\ \ell_0+\dots+\ell_j=L}}\
	\sum_{\substack{m_0,\dots,m_j \in \{0,1,\dots\}\\ m_0+\dots+m_j=M}}\
	\prod_{r=0}^{j} \left(\binom{\ell_r+m_r}{\ell_r,m_r}\right)^2\, .
\end{equation}
Then we have $A(N,j) = K(N,N,j)$.
Let us define the generating function
\begin{equation}
	\kappa(w,x,y)\, =\, \sum_{j=0}^{\infty} \sum_{L=0}^{\infty} \sum_{M=0}^{\infty} w^j x^L y^M K(L,M,j)\, ,
\end{equation}
for  $x,y,w \in \C$.
Then for $w,x,y \in [0,\infty)$, we have
\begin{equation}
\label{eq:betaSecondFormula}
	\kappa(w,x,y)\, =\, \sum_{j=0}^{\infty}  w^j \left(\sum_{\ell=0}^{\infty}\ \sum_{m=0}^{\infty} x^{\ell} y^{m} \left(\binom{\ell+m}{\ell,m}\right)^2\right)^{j+1}\, ,
\end{equation}
by the Tonelli theorem. (See for example, Section 2.5 of Folland \cite{Folland}.) So the left-hand-side
diverges if and only if the right-hand-side diverges. If neither diverges, the numerical equality 
is valid.
Also, if neither diverges for a choice of $x=x_1$, $y=y_1$, $w=w_1$ for $x_1,y_1,w_1 \in [0,\infty)$, then for any complex numbers $x,y,w \in \C$ such that 
$|x|\leq x_1$, $|y|\leq y_1$ and $|z|\leq z_1$, the two sides of (\ref{eq:betaSecondFormula}) are convergent 
and there is equality of the two sides  by Fubini's theorem.
Given $|x|,|y| \in [0,1/4)$, let us define
\begin{equation}
\label{eq:gamma2def}
	\varkappa^{(2)}(x,y)\, =\, \sum_{\ell=0}^{\infty} \sum_{m=0}^{\infty} x^\ell y^m \left(\binom{\ell+m}{\ell,m}\right)^2\, ,
\end{equation}
so that, according to (\ref{eq:betaSecondFormula}),
\begin{equation}
\label{eq:betaSecondFormula2}
	\kappa(w,x,y)\, =\, \sum_{j=0}^{\infty}  w^j \left(\varkappa^{(2)}(x,y)\right)^{j+1}\, .
\end{equation}
The reason we have denoted the function in (\ref{eq:gamma2def}) as $\varkappa^{(2)}$, with the superscript index $2$, is that
it bears a relation to a more famous generating function, which we call $\varkappa^{(1)}$:
\begin{equation}
	\forall x,y \in \left[0\, ,\ 1/2\right)\, ,\ \text{ we define }\ \varkappa^{(1)}(x,y)\, =\, \sum_{\ell,m=0}^{\infty} \binom{\ell+m}{\ell,m} x^{\ell} y^m\, ,
\end{equation}
using the multinomial notation for the binomial coefficient 
$\binom{\ell+m}{\ell,m} = \frac{(\ell+m)!}{\ell!\, m!}$.
We have 
\begin{equation}
	\varkappa^{(1)}(x,y)\, =\, \sum_{n=0}^{\infty} (x+y)^n\, =\, \frac{1}{1-(x+y)}\, .
\end{equation}
The function that we have called $\varkappa^{(1)}$ is among the first examples in many textbooks on generating functions. See, for example, Pemantle and Wilson's Example 2.2.2 in \cite{PemantleWilson}.

The relation between $\varkappa^{(1)}$ and $\varkappa^{(2)}$ is as follows.
Taking the product of $\varkappa^{(1)}(x,y)$ and $\varkappa^{(1)}(u,v)$ and then multiplying $x$ by $\xi$ and $u$ by $\xi^{-1}$,
and multiplying $y$ by $\zeta$ and $v$ by $\zeta^{-1}$,  we have
\begin{equation}
	\varkappa^{(1)}(x\xi,y\zeta) \cdot \varkappa^{(1)}\left(\frac{u}{\xi}\, ,\ \frac{v}{\zeta}\right)\, =\, \sum_{\ell,m,k,n=0}^{\infty} \binom{\ell+m}{\ell,m} \binom{k+n}{k,n}
	x^\ell y^m u^k v^n \xi^{\ell-k} \zeta^{m-n}\, .
\end{equation}
Therefore, integrating over unit circles, 
and using the residue method:
\begin{equation}
\label{eq:gamma2contour}
	\varkappa^{(2)}(xu,yv)\, =\, \oint_{\mathcal{C}(0;1)} \oint_{\mathcal{C}(0;1)} \varkappa^{(1)}(x\xi,y\zeta) \cdot \varkappa^{(1)}\left(\frac{u}{\xi}\, ,\ \frac{v}{\zeta}\right)\, \frac{d\xi}{2\pi i \xi}\, \cdot
	\frac{d\zeta}{2\pi i \zeta}\, .
\end{equation}
\begin{definition}
Let $\mathcal{C}(z_0;r)$ denote the circle in the complex plane $\{z\, :\, |z-z_0|=r\}$,
which may be parametrized as $z = z_0 + r e^{i\theta}$ for $-\pi<\theta\leq \pi$.
\end{definition}
One can perform the integral to calculate the following.
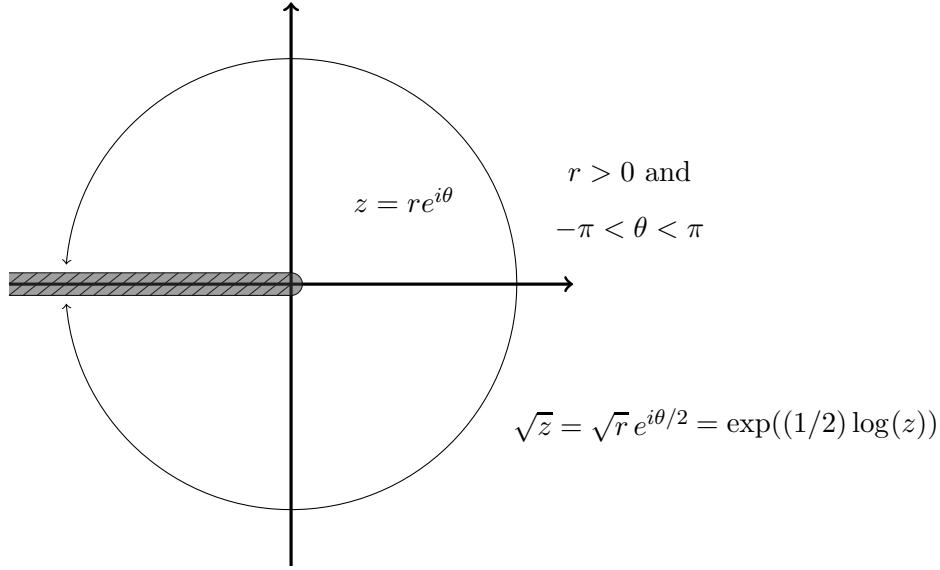
\begin{figure}
\begin{center}
\begin{tikzpicture}[xscale=1.5,yscale=1.5]
	\draw[very thick,->] (-2.5,0) -- (2.5,0);
	\draw[very thick,->] (0,-2.5) -- (0,2.5);
	\foreach \t in {-2.5,-2.375,-2.25,...,-0.25}{
		\draw[thin] (\t,-0.1) -- +(0.25,0.2);}
	\draw[thin] (-0.125,-0.1) -- (0,0);
	\draw[thin] (-2.5,0) -- (-2.375,0.1);
	\draw (-2.5,0.1) -- (0,0.1) arc (90:-90:0.1cm) --  (-2.5,-0.1);
	\fill[black!75!white, opacity=0.5] (-2.5,0.1) -- (0,0.1) arc (90:-90:0.1cm) --  (-2.5,-0.1) -- cycle;
	\draw[<->] (-175:2cm) arc (-175:175:2cm);
	\draw (1,0.75) node[] {$z=r e^{i\theta}$};
	\draw (3,1) node[] {$r>0$ and};
	\draw (3,0.5) node[] {$-\pi<\theta<\pi$};
	\draw (2.75,-1.25) node[] {$\sqrt{z}=\sqrt{r}\, e^{i\theta/2}$};
	\draw (3.5,-1.25) node[right] {$=\exp((1/2) \log(z))$};
%	\fill[black!75!white, opacity=0.5] (1.5,0) -- (1.5,0.1) arc (90:-90:0.1cm) -- cycle;
%	\draw (0.65,0.1) -- (0.65,-0.1) node[below] {\small $c_1$};
%	\draw (1.6,0.1) -- (1.6,-0.1) node[below] {\small $c_2$};
%	\draw (0,0) circle (2cm);
%	\fill (0.45,0) circle (2pt) node[above=1pt] {\small $a_1$};
%	\fill (1.8,0) circle (2pt) node[above=1pt] {\small $a_2$};
%	\begin{scope}[xshift=2.5cm]
%	\draw[line width=0.3cm, black!75!white, opacity=0.5] (0.75,0) -- (1.5,0);
%	\foreach \t in {0.65,0.7,0.75,...,1.6}{
%		\draw[thin] (\t,-0.1) -- (\t,0.1);}
%	\fill[black!75!white, opacity=0.5] (0.75,0) -- (0.75,0.1) arc (90:270:0.1cm) -- cycle;
%	\fill[black!75!white, opacity=0.5] (1.5,0) -- (1.5,0.1) arc (90:-90:0.1cm) -- cycle;
%	\draw (0.65,0.1) -- (0.65,-0.1) node[below] {\small $d_1$};
%	\draw (1.6,0.1) -- (1.6,-0.1) node[below] {\small $d_2$};
%	\fill (0.35,0) circle (2pt) node[above=1pt] {\small $b_1$};
%	\fill (1.9,0) circle (2pt) node[above=1pt] {\small $b_2$};
%	\end{scope}
%%	\fill (0,0) circle (2pt);
%	\draw (45:2cm) node[rotate=135] {\small $\boldsymbol{>}$};
\end{tikzpicture}
\caption{\textsl{Schematic for the standard square-root function on the open domain of $\C$ with branch cut along the non-positive
real axis: $\mathcal{D}_1 = \{r e^{i\theta}\, :\, r>0\, ,\ -\pi<\theta<\pi\}$. For $z \in \mathcal{D}_1$
we use the formula
$\sqrt{z} = \exp((1/2) \int_0^1 (z-1)(1+t(z-1))^{-1}\, dt)$. \label{fig:sqrt}
}}
\end{center}
\end{figure}
\begin{lemma}
\label{lem:gamma2}
Given $|x|,|y| \in [0,1/4)$, 
the series in (\ref{eq:gamma2def}) converges and 
\begin{equation}
\label{eq:gamma2formula}
	\varkappa^{(2)}(x,y)\, =\, \left(\sqrt{\left(1-\left(x+y\right)\right)^2 - 4 xy}\right)^{-1}\, .
\end{equation}
\end{lemma}
\begin{remark}
\label{rem:sqrt}
In our paper,
we use the standard branch cut for the natural logarithm along the non-positive
real axis. So the domain is $\mathcal{D}_1 = \{r e^{i\theta}\, :\, r>0\, ,\ -\pi<\theta<\pi\}$
and for each $z \in \mathcal{D}_1$, the formula is 
$\log(z) = \int_0^1 (z-1)(1+(z-1)t)^{-1}\, dt$. Then, for each $z \in \mathcal{D}_1$, we use $\sqrt{z} = \exp((1/2) \log(z))$.
See Figure \ref{fig:sqrt} for the schematic.
In the decomposition of $x=r e^{i\theta}$, we refer to the term $e^{i\theta}$ as the ``phase.''
Some complex analysis textbooks call this the argument as in Section 3.4 of \cite{Ahlfors}.
But we reserve the term ``argument'' for the input of other functions. 
For example, given a term of the form $\sqrt{f(x)}$ we will refer to $f(x)$ as the ``argument of the square-root function.''
\end{remark}

The result of the calculation in the lemma is an instance related to the fact that the diagonal of a rational bivariate generating function
is an algebraic univariate generating function. (And it is generically not rational.)
This is proved in Section 2.4 of Pemantle and Wilson,
credited to Furstenberg and to Hautus and Klarner.

The residue method may be used to derive formulas such as this, when they are unknown.
This is opposed to proof methods for conjectural formulas, which are sometimes more automatic, such as induction.
For this example, we do the contour integral as an exercise in Appendix \ref{sec:gamma2}.
But once the formula is ``known,'' whether or not it is {\em proved},
there are often simpler combinatorial methods that prove the formula.
The methods are often simpler because they bypass the need for contour integrals.
Here we present such an alternative proof.

\begin{proofof}{\bf Proof of Lemma \ref{lem:gamma2}:}
Let the right-hand-side of equation (\ref{eq:gamma2formula}) be denoted by $g^{(2)}(x,y)$. Then we may rewrite
\begin{equation}
	g^{(2)}(x,y)\, =\, \frac{1}{1-(x+y)}\, \left(1- \frac{4xy}{\left(1-(x+y)\right)^2}\right)^{-1/2}\, .
\end{equation}
Therefore, by Newton's version of the binomial theorem, this gives
\begin{align}
	g^{(2)}(x,y)\, &=\, \frac{1}{1-(x+y)}\, \sum_{n=0}^{\infty} \frac{(-1/2)_n}{n!} \frac{(-1)^n 4^n x^n y^n}{\left(1-(x+y)\right)^{2n}}\\
	&=\, \sum_{n,k=0}^{\infty} \frac{(-1/2)_n}{n!} \cdot \frac{(-2n-1)_k}{k!} (-1)^{n+k} 4^n x^n y^n (x+y)^k\, ,
\end{align}
where the Pochhammer symbol, falling factorial, is 
\begin{equation}
\label{eq:PochhammerFallingFactorial}
	(z)_n\, =\, z(z-1)(z-2)\cdots (z-n+1)\, =\, \prod_{k=0}^{n-1} (z-k)\, .
\end{equation}
Then, by the regular binomial formula, we have
\begin{equation}
	g^{(2)}(x,y)\, 
	=\, \sum_{n,k=0}^{\infty} \sum_{j=0}^{k} \frac{(-1/2)_n}{n!} \cdot \frac{(-2n-1)_k}{k!} \binom{k}{j} (-1)^{n+k} 4^n x^{n+j} y^{n+k-j}\, .
\end{equation}
Therefore, the lemma will be proved if we check that
\begin{equation}
\label{eq:LeftMostBridge}
	\sum_{n=0}^{\infty} \frac{(-1/2)_n}{n!} \cdot \frac{(-2n-1)_k}{k!} \binom{k}{j} (-1)^{n+k} 4^n \Bigg|_{\substack{j=\ell-n\\k-j=m-n}}\,
	=\, \left(\binom{\ell+m}{\ell,m}\right)^2\, ,
\end{equation}
for each $\ell,m \in \{0,1,\dots\}$.
But using well-known combinatorial formulas for the Pochhammer symbols of $-1/2$ and of negative integers, which we review in Appendix \ref{sec:Poc}, this is equivalent to checking
the multinomial identity:
\begin{equation}
\label{eq:gamma2combinatorial}
	\sum_{n=0}^{\infty} \binom{\ell+m}{n,n,\ell-n,m-n}\, =\, \left(\binom{\ell+m}{\ell,m}\right)^2\, .
\end{equation}
(We use the standard convention that if one of the indices in the subscript of a binomial or multinomial coefficient is outside the range
of $0$ to the superscript index, then that coefficient equals 0 by definition.)
Equation (\ref{eq:gamma2combinatorial}) may be proved as follows.
The key is to use {\em distinguishable permutations}.

Draw $\ell+m$ black dots on the number line from $1$ to $\ell+m$. Above each dot place a circle with either the top or bottom shaded in, such that if 
$a$ is the subset of those $r\in\{1,\dots,\ell+m\}$
whose circle above it has the top half shaded, then $|a|=\ell$ and $|a^{\complement}|=m$.
Below each black dot, place a circle with either the left or right shaded in, such that if $\mathcal{B}$ is the set of 
those $r\in\{1,\dots,\ell+m\}$
whose circle below it has the left half shaded, then $|\mathcal{B}|=m$ and $|\mathcal{B}^{\complement}|=\ell$.
We show a particular example of this mapping in the following displayed equation:
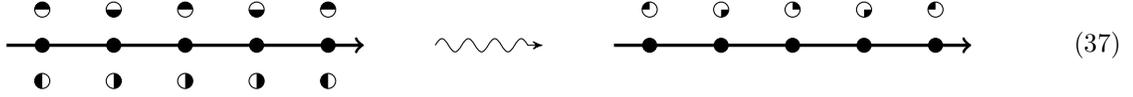
\begin{equation}
\label{eq:CombinatorialExample}
\begin{minipage}{13cm}	\begin{tikzpicture}[xscale=0.95,yscale=0.95]
		\draw[very thick,->] (0.5,0) -- (5.5,0);
		\foreach \i in {1,...,5}{
			\fill (\i,0) circle (3pt);}
		\foreach \i in {1,3,5}{
			\draw (\i,0.5) circle (3pt);
			\fill (\i,0.5) -- +(0:3pt) arc (0:180:3pt) -- cycle;}
		\foreach \i in {2,4}{
			\draw (\i,0.5) circle (3pt);
			\fill (\i,0.5) -- +(180:3pt) arc (180:360:3pt) -- cycle;}
		\foreach \i in {2,3,4}{
			\draw (\i,-0.5) circle (3pt);
			\fill (\i,-0.5) -- +(-90:3pt) arc (-90:90:3pt) -- cycle;}
		\foreach \i in {1,5}{
			\draw (\i,-0.5) circle (3pt);
			\fill (\i,-0.5) -- +(90:3pt) arc (90:270:3pt) -- cycle;}
	\begin{scope}[xshift=8.5cm]
		\draw[very thick,->] (0.5,0) -- (5.5,0);
		\foreach \i in {1,...,5}{
			\fill (\i,0) circle (3pt);}
		\foreach \i in {1,5}{
			\draw (\i,0.5) circle (3pt);
			\fill (\i,0.5) -- +(90:3pt) arc (90:180:3pt) -- cycle;}
		\foreach \i in {2,4}{
			\draw (\i,0.5) circle (3pt);
			\fill (\i,0.5) -- +(270:3pt) arc (270:360:3pt) -- cycle;}
		\foreach \i in {3}{
			\draw (\i,0.5) circle (3pt);
			\fill (\i,0.5) -- +(0:3pt) arc (0:90:3pt) -- cycle;}
	\end{scope}
	\draw[decorate, decoration={coil, aspect=0}, -stealth'] (6.5,0) -- (8,0);
	\end{tikzpicture}
\end{minipage}
\end{equation}
Then, if we let $n = |a\cap \mathcal{B}|$ we have $|a \cap \mathcal{B}^{\complement}|=\ell-n$
and $|a^{\complement} \cap \mathcal{B}|=m-n$, as well. Therefore, $|a^{\complement}\cap\mathcal{B}^{\complement}|$
is equal to $\ell+m-|a\cap \mathcal{B}|-|a \cap \mathcal{B}^{\complement}|-|a^{\complement} \cap \mathcal{B}|$,
which is $n$. The number of ways of choosing one quadrant to shade in -- (top,left), (top,right), (bottom,left), (bottom,right)  --
with a prescribed number of each $n$, $\ell-n$, $m-n$, $n$ is $N(\ell,m,n)$ where
\begin{equation}
	N(\ell,m,n)\ =\, \binom{\ell+m}{n,\ell-n,m-n,n}\, .
\end{equation}
Summing over all choices of $n$ gives all possibilities for choosing the two sets $a$ and $\mathcal{B}$.
Since they are chosen independently, that is equal to $\widetilde{N}(\ell,m) \widetilde{N}(\ell,m)$ where 
$\widetilde{N}(\ell,m)$ is the number of ways of choosing a set of cardinality $\ell$ as a subset of a larger set of cardinality
$\ell+m$, which is
\begin{equation}
\widetilde{N}(\ell,m)\, =\, \binom{\ell+m}{\ell,m}\, .
\end{equation}
Since we have proved that $\sum_{n} N(\ell,m,n)$ is equal to $\widetilde{N}(\ell,m) \widetilde{N}(\ell,m)$, this  proves equation (\ref{eq:gamma2combinatorial}).
\end{proofof}
\begin{remark}
\label{rem:Pinsky1}
The picture as in (\ref{eq:CombinatorialExample}) could have been made clearer if we had used multiple colors, as we did in a first draft of our preprint. But
we wished to obtain another picture which could be more easily printed in black-and-white. The combinatorial object is the same as one that was already
in Pinsky's paper in the first step of his proof of Lemma 2 \cite{Pinsky1}.
\end{remark}

Because of Lemma \ref{lem:gamma2} and equation (\ref{eq:betaSecondFormula}), we have
\begin{equation}
\label{eq:beta}
	\kappa(w,x,y)\, =\, \frac{\varkappa^{(2)}(x,y)}{1-w \cdot \varkappa^{(2)}(x,y)}\, =\, \frac{1}{\left(\left(1-\left(x+y\right)\right)^2-4xy\right)^{1/2}-w}\, .
\end{equation}
Now define another generating function
\begin{equation}
	\alpha(w,z)\, =\, \sum_{j,N=0}^{\infty} w^j z^N A(N,j)\, =\, \sum_{j,N=0}^{\infty} w^j z^N B(N,N,j)\, ,
\end{equation}
which is the generating function for the combinatorial numbers we want, $A(N,j)$
for $N,j \in \{0,1,\dots\}$.

}

\section{The diagonal (contour integral) method}
\label{sec:PinskyDiagonal}

We now apply the diagonal method, based on a complex  contour integral.
Firstly, we obtain
\begin{equation}
\label{eq:alphaIntegral}
	\oint_{\mathcal{C}(0;1)} \beta\left(w\, ,\ x\xi\, ,\ \frac{y}{\xi}\right)\, \frac{d\xi}{2\pi i \xi}\, =\, \alpha(w,xy)\, ,
\end{equation}
for sufficiently small $x$, $y$ and $w$.
To state it again with the detailed formula, we have
\begin{equation}
\label{eq:alphaIntegralOneHalf}
	\alpha(w,xy)\,  =\,
\oint_{\mathcal{C}(0;1)}  \frac{1}{\left(\left(1-\left(x \xi+y \xi^{-1}\right)\right)^2-4xy\right)^{1/2}-w}\, \cdot \frac{d\xi}{2\pi i \xi}\, .
\end{equation}
For simplicity, we will specialize this to $y=x$. That gives
\begin{equation}
\label{eq:alphaIntegralThreeFourths}
	\alpha(w,x^2)\,  =\,
\oint_{\mathcal{C}(0;1)}  \frac{1}{\left(\left(1-x\cdot\left(\xi+\xi^{-1}\right)\right)^2-4x^2\right)^{1/2}-w}\, \cdot \frac{d\xi}{2\pi i \xi}\, .
\end{equation}
This may then be further simplified as 
\begin{equation}
\label{eq:alphaIntegralFourFifths}
	\alpha(w,x^2)\,  =\,
\oint_{\mathcal{C}(0;1)}  \frac{1}{\Big(\big(1-x\cdot (\xi+2+\xi^{-1})\big)\big(1-x\cdot (\xi-2+\xi^{-1})\big)\Big)^{1/2}-w}\, \cdot \frac{d\xi}{2\pi i \xi}\, .
\end{equation}
We will do a contour integral for $\xi$ as in Figure \ref{fig:contour},
which requires analytically continuing this function from $\xi \in \mathcal{C}(0;1)$.
The argument of the square-root in equation (\ref{eq:alphaIntegralFourFifths}) is closely connected
to the quartic function of $\xi$
\begin{equation}
\label{eq:calQ1Def}
\begin{split}
	\mathcal{Q}_1(x;\xi)\, 
	&=\, \big(\xi - x\cdot(\xi^2+1)\big)^2-4x^2\xi^2\\ 
	&=\, \big(x \cdot(\xi^2+1) - (1+2x)\xi\big)\big(x \cdot(\xi^2+1) - (1-2x)\xi\big)\, .
\end{split}
\end{equation}
More precisely, we can rewrite equation (\ref{eq:alphaIntegralFourFifths}) as
\begin{equation}
\label{eq:alphaIntegralFiveSixths}
	\alpha(w,x^2)\,  =\,
\oint_{\mathcal{C}(0;1)}  \frac{1}{\big(\mathcal{Q}_1(x;\xi)/\xi^2\big)^{1/2}-w}\, \cdot \frac{d\xi}{2\pi i \xi}\, ,
\end{equation}
because, for $\xi \neq 0$ we have
\begin{equation}
\label{eq:sqrtArgument}
	\big(1-x\cdot (\xi+2+\xi^{-1})\big)\big(1-x\cdot (\xi-2+\xi^{-1})\big)\, 
	=\, \frac{\mathcal{Q}_1(x;\xi)}{\xi^2}\, .
\end{equation}
We will want to rationalize the denominator of the integrand. Therefore, let us also introduce
another quartic function of $\xi$, 
\begin{equation}
\label{eq:calQ2Def}
\begin{split}
	\mathcal{Q}_2(x,w;\xi)\, 
	&=\, \mathcal{Q}_1(x;\xi)-w^2 \xi^2\\
	&=\, \big(\xi - x\cdot(\xi^2+1)\big)^2-(4x^2+w^2)\xi^2\, .
\end{split}
\end{equation}
If we assume that $x$ and $w$ are both in $[0,\infty)$, then we may rewrite this as 
\begin{equation}
\label{eq:calQ2Def2}
	\mathcal{Q}_2(x,w;\xi)\, 
	=\, \left(x \cdot(\xi^2+1) - \left(1+\sqrt{4x^2+w^2}\right)\xi\right)
	\left(x \cdot(\xi^2+1) - \left(1-\sqrt{4x^2+w^2}\right)\xi\right)\, .
\end{equation}
Now we will begin to simplify the polynomials $\mathcal{Q}_2(x;\cdot)$ and $\mathcal{Q}_2(x,w;\cdot)$.
First we may a quick definition to set notation.
\begin{definition}
Let $\mathcal{U}(z_0;r)$ denote the open unit disk in the complex plane $\{z\, :\, |z-z_0|<r\}$.
Let $\overline{\mathcal{U}}(z_0;r)$ denote its closure $\{z\, :\, |z-z_0|\leq r\}$.
\end{definition}
The following four lemmas will be proved in in Appendix \ref{sec:quadratic}. They all rely on the quadratic formula
for roots of 1-variable quadratic equations.
\begin{figure}
\begin{center}
\begin{tikzpicture}[xscale=1.5,yscale=1.5]
	\draw[very thick,->] (-2.5,0) -- (5,0);
	\draw[very thick,->] (0,-2.5) -- (0,2.5);
	\draw[line width=0.3cm, black!75!white, opacity=0.5] (0.75,0) -- (1.5,0);
	\foreach \t in {0.65,0.7,0.75,...,1.6}{
		\draw[thin] (\t,-0.1) -- (\t,0.1);}
	\fill[black!75!white, opacity=0.5] (0.75,0) -- (0.75,0.1) arc (90:270:0.1cm) -- cycle;
	\fill[black!75!white, opacity=0.5] (1.5,0) -- (1.5,0.1) arc (90:-90:0.1cm) -- cycle;
	\draw (0.65,0.1) -- (0.65,-0.1) node[below] {\small $c_1$};
	\draw (1.6,0.1) -- (1.6,-0.1) node[below] {\small $c_2$};
	\draw (0,0) circle (2cm);
	\fill (0.45,0) circle (2pt) node[above=1pt] {\small $a_1$};
	\fill (1.8,0) circle (2pt) node[above=1pt] {\small $a_2$};
	\begin{scope}[xshift=2.5cm]
	\draw[line width=0.3cm, black!75!white, opacity=0.5] (0.75,0) -- (1.5,0);
	\foreach \t in {0.65,0.7,0.75,...,1.6}{
		\draw[thin] (\t,-0.1) -- (\t,0.1);}
	\fill[black!75!white, opacity=0.5] (0.75,0) -- (0.75,0.1) arc (90:270:0.1cm) -- cycle;
	\fill[black!75!white, opacity=0.5] (1.5,0) -- (1.5,0.1) arc (90:-90:0.1cm) -- cycle;
	\draw (0.65,0.1) -- (0.65,-0.1) node[below] {\small $d_1$};
	\draw (1.6,0.1) -- (1.6,-0.1) node[below] {\small $d_2$};
	\fill (0.35,0) circle (2pt) node[above=1pt] {\small $b_1$};
	\fill (1.9,0) circle (2pt) node[above=1pt] {\small $b_2$};
	\end{scope}
%	\fill (0,0) circle (2pt);
	\draw (45:2cm) node[rotate=135] {\small $\boldsymbol{>}$};
\end{tikzpicture}
\caption{\textsl{A schematic picture -- it is not to scale but it does accurately depict the order of points -- for the branch cuts and poles of the integral in (\ref{eq:alphaIntegralSevenEighths}).
The configuration of points and branch cuts is shown assuming $u,t \in (0,1)$ are sufficiently small.
The poles are at $a_1(x,w)$, $a_2(x,w)$, $b_1(x,w)$ and $b_2(x,w)$.
The branch cuts are the intervals $[c_1(x),c_2(x)]$ and $[d_1(x),d_2(x)]$.
\label{fig:contour}
}}
\end{center}
\end{figure}
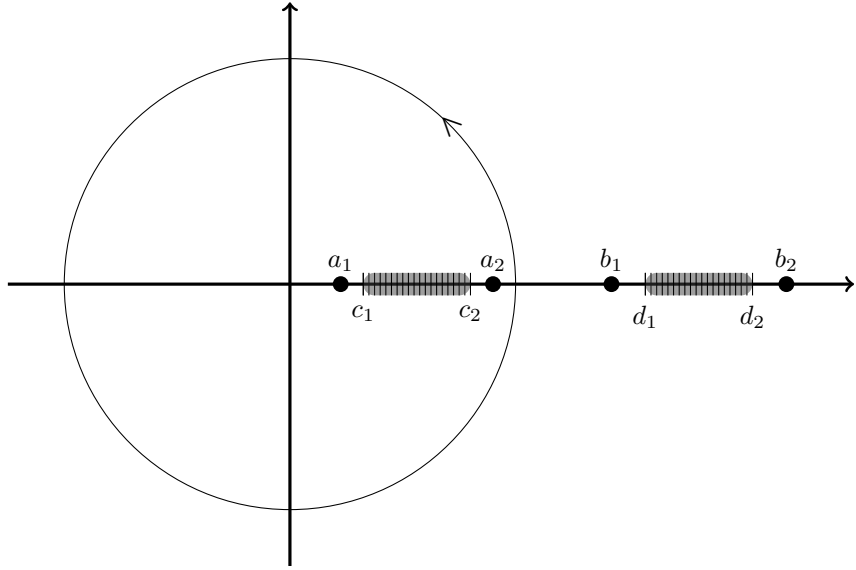
\begin{lemma}
\label{lem:calQ1roots}
Suppose that we have a number $x \in (0,1/4)$. Then, defining
\begin{equation}
\begin{split}
	c_1(x)\, &=\, \frac{1+2x-\sqrt{1+4x}}{2x}\, ,\qquad c_2(x)\, =\, \frac{1-2x-\sqrt{1-4x}}{2x}\, ,\\[5pt]
	d_1(x)\, &=\, \frac{1-2x+\sqrt{1-4x}}{2x}\, ,\qquad d_2(x)\, =\, \frac{1+2x+\sqrt{1+4x}}{2x}\, ,
\end{split}
\end{equation}
we have
\begin{equation}
\label{eq:calQ1roots}
	\mathcal{Q}_1(x;\xi)\, =\, x^2 \big(\xi  - c_1(x)\big)\big(\xi-c_2(x)\big)\big(\xi-d_1(x)\big)\big(\xi-d_2(x)\big)\, .
\end{equation}
\end{lemma}
Note in particular that
\begin{equation}
\label{eq:Q1inversive}
	c_1(x)\cdot d_2(x)\, =\, c_2(x) \cdot d_1(x)\, =\, 1\, ,
\end{equation}
which can also be seen since they are the roots of the self-inversive quadratic polynomials of $\xi$ in the ultimate formula
in (\ref{eq:calQ1Def}). 
\begin{lemma}
\label{lem:calQ2roots}
Suppose that we have a number $x \in (0,1/4)$ and another number $w$ satisfying
\begin{equation}
	0\, \leq\, w\, \leq\, \sqrt{1-4x}\, .
\end{equation}
Then, defining
\begin{equation}
\begin{split}
	a_1(x,w)\, &=\, \frac{1+\sqrt{4x^2+w^2}-\sqrt{1+w^2+2\sqrt{4x^2+w^2}}}{2x}\, ,\\[5pt]
	a_2(x,w)\, &=\, \frac{1-\sqrt{4x^2+w^2}-\sqrt{1+w^2-2\sqrt{4x^2+w^2}}}{2x}\, ,\\[5pt]
	b_1(x,w)\, &=\, \frac{1-\sqrt{4x^2+w^2}+\sqrt{1+w^2-2\sqrt{4x^2+w^2}}}{2x}\, ,\\[5pt]
	b_2(x,w)\, &=\, \frac{1+\sqrt{4x^2+w^2}+\sqrt{1+w^2+2\sqrt{4x^2+w^2}}}{2x}\, ,
\end{split}
\end{equation}
we have
\begin{equation}
\label{eq:Q2factorization}
	\mathcal{Q}_2(x,w;\xi)\, =\, x^2\cdot \big(\xi  - a_1(x,w)\big)\big(\xi-a_2(x,w)\big)\big(\xi-b_1(x,w)\big)\big(\xi-b_2(x,w)\big)\, .
\end{equation}
\end{lemma}
Of course,
we also have
\begin{equation}
\label{eq:Q2inversive}
	a_1(x,w)\cdot b_2(x,w)\, =\, a_2(x,w) \cdot b_1(x,w)\, =\, 1\, ,
\end{equation}
because these are the roots of the self-inversive quadratic polynomials of $\xi$ 
in (\ref{eq:calQ2Def}). 
Let us note the following ordering of the roots for the two quartic equations,
as in Figure \ref{fig:contour}.
\begin{lemma}
\label{lem:rootsOrder}
If we have $x \in (0,1/4)$ and $w \in (0,\sqrt{1-4x})$ then we have
\begin{equation}
	0<a_1(x,w)<c_1(x)<c_2(x)<a_2(x,w)<1<b_1(x,w)<d_1(x)<d_2(x)<b_2(x,w)\, .
\end{equation}
\end{lemma}
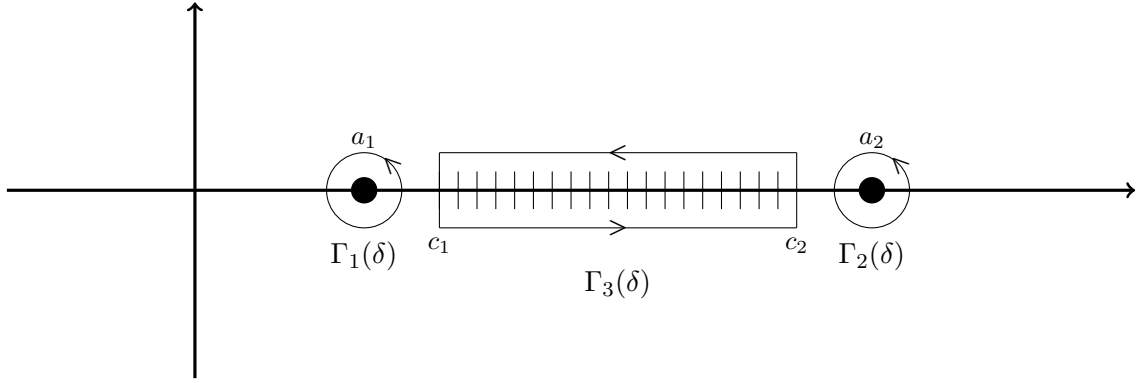
\begin{figure}
\begin{center}
\begin{tikzpicture}[xscale=5,yscale=5]
	\draw[very thick,->] (-0.5,0) -- (2.5,0);
	\draw[very thick,->] (0,-0.5) -- (0,0.5);
	\foreach \t in {0.65,0.7,0.75,...,1.6}{
		\draw[thin] (\t,-0.05) -- (\t,0.05);}
	\draw (0.65,0.1) -- (0.65,-0.1) node[below] {\small $c_1$};
	\draw (0.65,-0.1) -- (1.6,-0.1);
	\draw (1.125,-0.1) node[rotate=0] {\scriptsize $\boldsymbol{>}$};
	\draw (1.6,0.1) -- (1.6,-0.1) node[below] {\small $c_2$};
	\draw (1.6,0.1) -- (0.65,0.1);
	\draw (1.125,0.1) node[rotate=0] {\scriptsize $\boldsymbol{<}$};
	\fill (0.45,0) circle (1pt) node[above=12pt] {\small $a_1$};
	\draw (0.45,0) circle (0.1cm);
	\draw (0.45,-0.175) node[] {$\Gamma_1(\delta)$};
	\draw (0.45,0) +(45:0.1cm) node[rotate=135] {\scriptsize $\boldsymbol{>}$};
	\fill (1.8,0) circle (1pt) node[above=12pt] {\small $a_2$};
	\draw (1.8,0) circle (0.1cm);
	\draw (1.8,-0.175) node[] {$\Gamma_2(\delta)$};
	\draw (1.8,0) +(45:0.1cm) node[rotate=135] {\scriptsize $\boldsymbol{>}$};
	\draw (1.125,-0.25) node[] {$\Gamma_3(\delta)$};
\end{tikzpicture}
\caption{\textsl{A schematic of the contours after deformation as used in the proof of Lemma \ref{lem:residue1}. The  positively oriented circular contours 
$\Gamma_1(\delta)$ and $\Gamma_2(\delta)$ are of radius $\delta$, centered at the poles $a_1(x,w)$
and $a_2(x,w)$. The branch cut along the real axis interval from $c_1(x)$ to $c_2(x)$ is enclosed by
the contour $\Gamma_3(\delta)$ which is the rectangle traversing the vertices: $c_1(x)-i\delta$, $c_2(x)-i\delta$, $c_2(x)+i\delta$ 
and $c_1(x)+i\delta$, in that order.
\label{fig:Deformation}
}}
\end{center}
\end{figure}
Let us introduce an analytic continuation of the function from (\ref{eq:sqrtArgument}),
$\left(\left(1-x\xi-x\xi^{-1}\right)-4x^2\right)^{1/2}$.
\begin{lemma}
\label{lem:sqrtContinuation}
Assume we have $x \in (0,1/4)$. Define the domain $\mathcal{D}_2(x)$ as 
\begin{equation}
	\mathcal{D}_2(x)\, =\, \mathcal{U}(0;d_1(x)) \setminus \overline{\mathcal{U}}(0;c_1(x))\, .
\end{equation}
Then for all $\xi \in \mathcal{D}_2(x)$, we have $\mathcal{Q}_1(x;\xi)/\xi^2$ is in $\mathcal{D}_1$
from Remark \ref{rem:sqrt}.
So, defining the function,
\begin{equation}
f(x;\cdot) : \mathcal{D}_2(x) \to \C\, ,
\end{equation}
by the formula
\begin{equation}
\label{eq:fDef}
	f(x;\xi)\, =\, \sqrt{\mathcal{Q}_2(x;\xi)/\xi^2}\, =\, \big(\left(1 - x \cdot ( \xi +\xi^{-1})\right)^2-4x^2\big)^{1/2}\, ,
\end{equation}
we have that $f(x;\cdot)$ is analytic.
\end{lemma}
We will prove Lemmas \ref{lem:calQ1roots}, \ref{lem:calQ2roots}, \ref{lem:rootsOrder} 
and \ref{lem:sqrtContinuation} in Appendix \ref{sec:quadratic}.
They are all elementary, and essentially follow by applying the quadratic formula.

We want to deform the contour from the integral in (\ref{eq:alphaIntegralFiveSixths})
as in Figure \ref{fig:Deformation}.
For equation (\ref{eq:alphaIntegralFiveSixths}) we rationalize the denominator.
From equation (\ref{eq:calQ2Def2}) and Lemma \ref{lem:sqrtContinuation}, we may 
rewrite that equation as
\begin{equation}
\label{eq:alphaIntegralSixSevenths}
	\alpha(w,x^2)\,  =\,
\oint_{\mathcal{C}(0;1)}  \frac{\xi\cdot f(x;\xi)+w\cdot \xi}
	{\mathcal{Q}_2(x,w;\xi)}\, \cdot \frac{d\xi}{2\pi i}\, .
\end{equation}
We multiplied numerator and denominator of the first fraction by $\xi^2$,
and then canceled one power from the numerator with the $\xi$ that had been
in the denominator of $d\xi/(2\pi i \xi)$, which is now absent.

To deform this as in as in Figure \ref{fig:Deformation},
we need a more optimal analytic continuation of the function $f(x;\cdot)$.
We start with the domain as in the right-hand picture in Figure \ref{fig:sqrtNew}.
\begin{lemma}
\label{lem:sqrtContinuationB}
Assume we have $x \in (0,1/4)$. Define the domain $\mathcal{D}_3(x)$ as 
\begin{equation}
	\mathcal{D}_3(x)\, =\, \C \setminus \big((-\infty,c_2(x)] \cup [d_1(x),\infty)\big)\, .
\end{equation}
Then define the  function
\begin{equation}
g(x,\cdot) : \mathcal{D}_3(x) \to \C\, ,
\end{equation}
by the formula
\begin{equation}
\label{eq:gDefinition}
	g(x;\xi)\, =\, x\, \cdot \sqrt{\xi-c_1(x)}\, \cdot \sqrt{\xi-c_2(x)}\,
\cdot \sqrt{d_1(x)-\xi}\, \cdot \sqrt{d_2(x)-\xi}\, ,
\end{equation}
where the square-root is as in Remark \ref{rem:sqrt}.
Then $g(x;\cdot)$ is analytic and 
\begin{equation}
\label{eq:gEQUALSf}
	\forall \xi \in \mathcal{D}_2(x)\cap \mathcal{D}_3(x)= \big(\mathcal{U}(0;d_1(x))\setminus \overline{\mathcal{U}}(0;c_1(x))\big)
	\setminus [-d_1(x),-c_1(x)]\, ,\
	\text{ we have }\ g(x;\xi)\, =\, \xi\cdot f(x;\xi)\, . 
\end{equation}
\end{lemma}
%But the branch cuts are each square-root type branch cuts. Therefore, wherever an even number of them overlap,
%one may eliminate the branch cut at that point by taking limits to resolve the value of the function.
%Hence, we obtain the following.
\begin{lemma}
\label{lem:sqrtContinuationC}
Assume we have $x \in (0,1/4)$. Define the domain $\mathcal{D}_4(x)$ as 
\begin{equation}
	\mathcal{D}_4(x)\, =\, \C \setminus \big([c_1(x),c_2(x)]\cup[d_1(x),d_2(x)]\big)\, .
\end{equation}
Then define the  function
\begin{equation}
\label{eq:tildegDef1}
	\widetilde{g}(x;\cdot) : \mathcal{D}_4(x) \to \C\, ,
\end{equation}
such that
\begin{equation}
\label{eq:tildegDef2}
	\forall \xi \in \mathcal{D}_3(x) = \Big(\C \setminus \big((-\infty,c_2(x)] \cup [d_1(x),\infty)\big)\Big)\, ,\ \text{ we have }\
	\widetilde{g}(x;\xi)\, =\, g(x;\xi)\, ,
\end{equation}
and
\begin{equation}
\label{eq:tildegDef3}
\forall \xi \in (-\infty,c_1(x)) \cup (d_2(x),\infty)\, ,\ 
\text{ we have }\ \widetilde{g}(x;\xi)\, =\, \lim_{\substack{\zeta \to \xi\\ \zeta \in  \C \setminus \R}}
 g(x;\zeta)\, .
\end{equation}
Then $\widetilde{g}(x;\cdot)$ is analytic and 
\begin{equation}
\label{eq:signCheck}
	\forall \xi \in \mathcal{D}_2(x)\, ,\
	\text{ we have }\ \widetilde{g}(x;\xi)\, =\, \xi \cdot f(x;\xi)\, .
\end{equation}
\end{lemma}

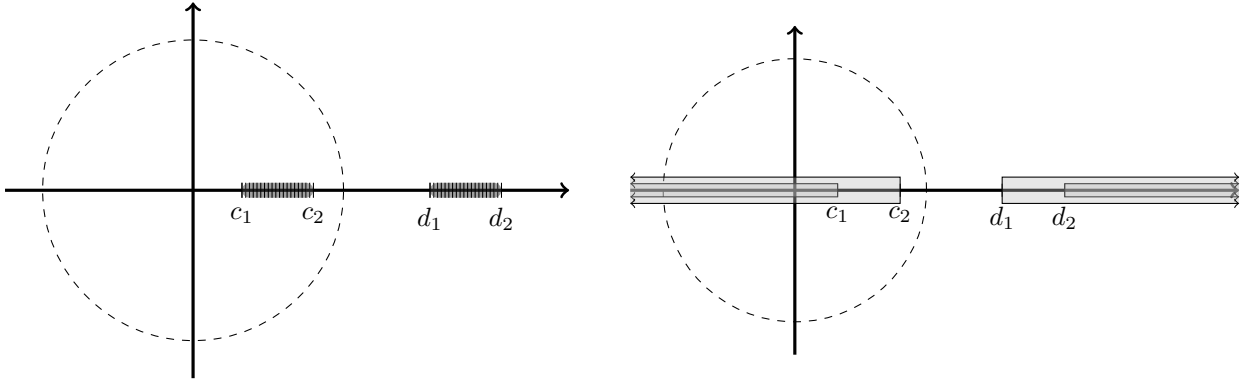
\begin{figure}
\begin{center}
\begin{tikzpicture}[xscale=1,yscale=1]
	\draw[very thick,->] (-2.5,0) -- (5,0);
	\draw[very thick,->] (0,-2.5) -- (0,2.5);
	\draw[line width=0.2cm, black!75!white, opacity=0.5] (0.75,0) -- (1.5,0);
	\foreach \t in {0.65,0.7,0.75,...,1.6}{
		\draw[thin] (\t,-0.1) -- (\t,0.1);}
	\fill[black!75!white, opacity=0.5] (0.75,0) -- (0.75,0.1) arc (90:270:0.1cm) -- cycle;
	\fill[black!75!white, opacity=0.5] (1.5,0) -- (1.5,0.1) arc (90:-90:0.1cm) -- cycle;
	\draw (0.65,0.1) -- (0.65,-0.1) node[below] {\small $c_1$};
	\draw (1.6,0.1) -- (1.6,-0.1) node[below] {\small $c_2$};
	\draw[dashed] (0,0) circle (2cm);
	\begin{scope}[xshift=2.5cm]
	\draw[line width=0.2cm, black!75!white, opacity=0.5] (0.75,0) -- (1.5,0);
	\foreach \t in {0.65,0.7,0.75,...,1.6}{
		\draw[thin] (\t,-0.1) -- (\t,0.1);}
	\fill[black!75!white, opacity=0.5] (0.75,0) -- (0.75,0.1) arc (90:270:0.1cm) -- cycle;
	\fill[black!75!white, opacity=0.5] (1.5,0) -- (1.5,0.1) arc (90:-90:0.1cm) -- cycle;
	\draw (0.65,0.1) -- (0.65,-0.1) node[below] {\small $d_1$};
	\draw (1.6,0.1) -- (1.6,-0.1) node[below] {\small $d_2$};
	\end{scope}
%	\fill (0,0) circle (2pt);
%	\draw (45:2cm) node[rotate=135] {\small $\boldsymbol{>}$};
\begin{scope}[xshift=8cm,xscale=0.875,yscale=0.875]
	\draw[dashed] (0,0) circle (2cm);
	\draw[very thick,->] (-2.5,0) -- (6.75,0);
	\draw[very thick,->] (0,-2.5) -- (0,2.5);
	\draw (0.65,0.1) -- (0.65,-0.1) node[below] {\small $c_1$};
	\fill[black!25!white,opacity=0.375] (-2.5,0.1) -- (0.65,0.1) -- (0.65,-0.1) -- (-2.5,-0.1) -- cycle;
	\draw[<->,thin] (-2.5,0.1) -- (0.65,0.1) -- (0.65,-0.1) -- (-2.5,-0.1);
	\draw (1.6,0.1) -- (1.6,-0.1) node[below] {\small $c_2$};
	\fill[black!25!white,opacity=0.375] (-2.5,0.2) -- (1.6,0.2) -- (1.6,-0.2) -- (-2.5,-0.2) -- cycle;
	\draw[<->,thin] (-2.5,0.2) -- (1.6,0.2) -- (1.6,-0.2) -- (-2.5,-0.2);
	\begin{scope}[xshift=2.5cm]
	\draw (1.6,0.1) -- (1.6,-0.1) node[below] {\small $d_2$};
	\fill[black!25!white,opacity=0.375] (4.25,0.1) -- (1.6,0.1) -- (1.6,-0.1) -- (4.25,-0.1) -- cycle;
	\draw[<->,thin] (4.25,0.1) -- (1.6,0.1) -- (1.6,-0.1) -- (4.25,-0.1);
	\draw (0.65,0.1) -- (0.65,-0.1) node[below] {\small $d_1$};
	\fill[black!25!white,opacity=0.375] (4.25,0.2) -- (0.65,0.2) -- (0.65,-0.2) -- (4.25,-0.2) -- cycle;
	\draw[<->,thin] (4.25,0.2) -- (0.65,0.2) -- (0.65,-0.2) -- (4.25,-0.2);
	\end{scope}
\end{scope}
\end{tikzpicture}
\caption{\textsl{The branch cuts for the function $\widetilde{g}(x;\cdot)$ are shown on the left,
which is an analytic continuation of the function $f(x;\xi)=\sqrt{\mathcal{Q}_1(\xi)/\xi^2}$
defined in a neighborhood of $\mathcal{C}(0;1)$.
In the figure on the right we show the {\em a priori} overlapping branch-cuts of the original
analytic continuation $g(x;\cdot)$.
\label{fig:sqrtNew}
}}
\end{center}
\end{figure}
The domain for the function $\widetilde{g}(x;\cdot)$ is as in the left-hand picture in Figure \ref{fig:sqrtNew}.
%Equation (\ref{eq:signCheck}) is a {\em bit} more precise than the more obvious relation
%$(f(x;\xi))^2 = (\widetilde{g}(x;\xi)/\xi)^2$. But it can be easily seen to be true at $\xi=1$.
We prove Lemmas \ref{lem:sqrtContinuationB} and \ref{lem:sqrtContinuationC} in 
Appendix \ref{sec:Deform}.
Then equation (\ref{eq:alphaIntegralSixSevenths}) may be rewritten as 
\begin{equation}
\label{eq:alphaIntegralSevenEighths}
	\alpha(w,x^2)\,  =\,
\oint_{\mathcal{C}(0;1)}  \frac{\widetilde{g}(x;\xi)+w\cdot \xi}
	{\mathcal{Q}_2(x,w;\xi)}\, \cdot \frac{d\xi}{2\pi i}\, ,
\end{equation}
where the roots of $\mathcal{Q}_2(x,w;\cdot)$ are $a_1(x;w)$, $a_2(x;w)$, $b_1(x;w)$ and $b_2(x;w)$,
and where the domain of $\widetilde{g}(x;\cdot)$ is all of $\C$ except for the union 
of the two intervals $[c_1(x),c_2(x)]$ and $[d_1(x),d_2(x)]$.
Thus we obtain the region  as in Figure \ref{fig:contour}.
This is the region where we may deform the contour $\mathcal{C}(0;1)$.

\section{Main result}
\label{sec:Main}
Using equations (\ref{eq:Q2factorization}) and (\ref{eq:alphaIntegralSevenEighths}), we know
\begin{equation}
\label{eq:alphaIntegralSevenEights}
	\alpha(w,x^2)\,  =\,
\oint_{\mathcal{C}(0;1)}  \frac{w\cdot \xi+\widetilde{g}(x;\xi)}
{x^2\cdot \big(\xi  - \mathcal{A}_1(x,w)\big)\big(\xi-\mathcal{A}_2(x,w)\big)\big(\xi-b_1(x,w)\big)\big(\xi-b_2(x,w)\big)}\, 
\cdot \frac{d\xi}{2\pi i}\, .
\end{equation}
We will deform the contour from $\mathcal{C}(0;1)$ into the open disk $\mathcal{U}(0;1)$ shrinking to contours around the poles
$\mathcal{A}_1(x,w)$ and $\mathcal{A}_2(x,w)$ and also around the
branch cut $[c_1(x),c_2(x)]$ as in Figure \ref{fig:Deformation}.
\begin{lemma}
\label{lem:residue1}
For $x,w \in (0,\infty)$ satisying $4x+w^2<1$ we have
\begin{equation}
	\alpha(w,x^2)\, =\, \mathcal{A}_1(x,w) + \mathcal{A}_2(x,w)\, ,
\end{equation}
for the function $\mathcal{A}_1(x,w)$ given by the residue theorem
\begin{equation}
\label{eq:A1formula}
\begin{split}
	\mathcal{A}_1(x,w)\, &=\, \frac{w\cdot a_1(x,w)+\widetilde{g}\big(x;a_1(x,w)\big)}
{x^2\cdot \big(a_1(x,w)-a_2(x,w)\big)\big(a_1(x,w)-b_1(x,w)\big)\big(a_1(x,w)-b_2(x,w)\big)}\\[5pt]
&\qquad
+\frac{w\cdot a_2(x,w)+\widetilde{g}\big(x;a_2(x,w)\big)}
{x^2\cdot \big(a_2(x,w)-a_1(x,w)\big)\big(a_2(x,w)-b_1(x,w)\big)\big(a_2(x,w)-b_2(x,w)\big)}\, ,
\end{split}
\end{equation}
and for the function $\mathcal{A}_2(x,w)$ given by the integral
\begin{equation}
\label{eq:A2formula}
	\mathcal{A}_2(x,w)\, =\, \frac{1}{\pi}\cdot \int_{c_1(x)}^{c_2(x)} \mathcal{I}_2(x,w;r)\, dr\, ,
\end{equation}
where
\begin{equation}
	\forall r \in [c_1(x),c_2(x)]\, ,\ \text{ we have }\ \mathcal{I}_2(x,w;r)\, =\, \frac{i\cdot \widetilde{g}_{+}(x;r)-i\cdot \widetilde{g}_-(x;r)}
	{2\mathcal{Q}_2(x,w;r)}\, , 
\end{equation}
for the functions $\widetilde{g}_+(x;\cdot)$ and $\widetilde{g}_-(x;\cdot)$ defined such that
\begin{equation}
	\forall r \in [c_1(x),c_2(x)]\, ,\ \text{ we have }\ \widetilde{g}_{\pm}(x;r)\, =\, \lim_{\epsilon \to 0^+} \widetilde{g}(x;r\pm i \epsilon)\, .
\end{equation}
\end{lemma}
This lemma is elementary and will be proved in Appendix \ref{sec:residue}.
Note that since $w\cdot \xi/\mathcal{Q}_2(x,\xi)$ is analytic for $\xi$ in a neighborhood
of $[c_1(x),c_2(x)]$, its contribution to the integral along $\Gamma_3(\delta)$ vanishes.
See Figure \ref{fig:Deformation}.
That is why it is not included in $\mathcal{I}_2(x,w;r)$.

Note that $\mathcal{A}_1(x,w)$ and $\mathcal{A}_2(x,w)$ are roots of $\mathcal{Q}_2(x,w;\xi)$ which is 
$\mathcal{Q}_1(x;\xi)-w^2 \xi^2$ as in (\ref{eq:calQ2Def}) and (\ref{eq:Q2factorization}). 
But $\big(\widetilde{g}(x;\xi)\big)^2 = \mathcal{Q}_1(x;\xi)$ 
by (\ref{eq:fDef}) and (\ref{eq:signCheck}).
Therefore, $w\cdot a_j(x,w)+\widetilde{g}(x;a_j(x,w))$ may be simplified for $j=1,2$
in the formula for $\mathcal{A}_1(x,w)$.
\begin{lemma} For $x,w \in (0,\infty)$ satisfying $4x+w^2<1$ we have
\begin{equation}
\label{eq:A1formula2}
\begin{split}
	\mathcal{A}_1(x,w)\, &=\, \frac{2w\cdot a_1(x,w)}
{x^2\cdot \big(a_1(x,w)-a_2(x,w)\big)\big(a_1(x,w)-b_1(x,w)\big)\big(a_1(x,w)-b_2(x,w)\big)}\\[5pt]
&\qquad
+\frac{2w\cdot a_2(x,w)}
{x^2\cdot \big(a_2(x,w)-a_1(x,w)\big)\big(a_2(x,w)-b_1(x,w)\big)\big(a_2(x,w)-b_2(x,w)\big)}\, ,
\end{split}
\end{equation}
or simplifying
\begin{equation}
\label{eq:A1formula3}
	a_1(x,w)\, =\, \frac{2w\cdot \big(b_1(x,w)\cdot b_2(x,w) - a_1(x,w) \cdot a_2(x,w)\big)}
{x^2\prod_{j=1}^{2} \prod_{k=1}^{2} \big(b_j(x,w)-a_k(x,w)\big)}\, .
\end{equation}
\end{lemma}
We will prove this lemma in Appendix \ref{sec:residue}.
Recall from equation (\ref{eq:Q2inversive}) that $b_2(x,w) = 1/a_1(x,w)$ and $b_1(x,w)=1/a_2(x,w)$.
From this one may simplify a bit more.
\begin{lemma} For $x,w \in (0,\infty)$ satisfying $4x+w^2<1$ we have
\begin{equation}
\label{eq:A1formula4}
\begin{split}
	\mathcal{A}_1(x,w)\, =\, \frac{2w\cdot a_1(x,w) \cdot a_2(x,w) \Big(1+ a_1(x,w) \cdot a_2(x,w)\Big)}
{x^2 \cdot\Big(1-a_1(x,w)\cdot a_2(x,w)\Big) \Big(1-\big(a_1(x,w)\big)^2\Big)
\Big(1-\big(a_2(x,w)\big)^2\Big)}\, .
\end{split}
\end{equation}
\end{lemma}
\begin{figure}
\begin{center}
\begin{tikzpicture}[xscale=4,yscale=4]
	\draw[dashed] (2,0) arc (0:15:2cm);
	\draw[dashed] (2,0) arc (0:-15:2cm);
	\draw[very thick,->] (-0.5,0) -- (2.5,0);
	\draw[very thick,->] (0,-0.675) -- (0,0.675);
	\draw (0.65,0.1) -- (0.65,-0.1) node[below] {\small $c_1$};
	\draw (1.6,0.1) -- (1.6,-0.1) node[below] {\small $c_2$};
	\draw (0.65,0) -- (1,0.5) -- (1.6,0);
	\draw (1,0.5) node[] {$\bullet$};
	\draw[] (0.65,0) +(0:0.2cm) arc (0:55:0.2cm);
	\draw (0.65,0) +(47.5:0.2cm) node[rotate=-45] {\small $<$};
	\draw[] (1.6,0) +(0:0.2cm) arc (0:140:0.2cm);
	\draw (1.6,0) +(135:0.2cm) node[rotate=45] {\small $<$};
	\draw[dotted] (1,0) -- (1,0.5);
	\draw (1,0.5) node[above] {$\xi = r + i \epsilon$};
	\draw (0.65,0) +(0:0.2cm) node[above left] {\small $\theta_1$};
	\draw (1.6,0) +(0:0.2cm) node[above left] {\small $\theta_2$};
\end{tikzpicture}
\caption{\textsl{Consideration of the phase of the function,  $\widetilde{g}(x;\xi) = \sqrt{\mathcal{Q}_1(x;\xi)}$, for $\xi=r+i\epsilon$
with $c_1(x)\leq r\leq c_2(x)$.
If we write $\widetilde{g}(x;\xi)$ as $e^{i \varphi} R$ for $R \geq 0$ and $\varphi \in \R$
then by ``phase'' we mean $e^{i\varphi}$. For $r \in \R \setminus \{c_1(x),c_2(x),d_1(x),d_2(x)\}$, and taking 
$\xi = r + i\delta$ the phase $e^{i\varphi}$
converges to a number $\Omega_+(r) \in \{1,i,-1,-i\}$ in the $\delta \to 0^+$ limit. Similarly for the limit of the phase from the lower half-plane $\Omega_-(r)$.
\label{fig:sqrtNew2}
}}
\end{center}
\end{figure}
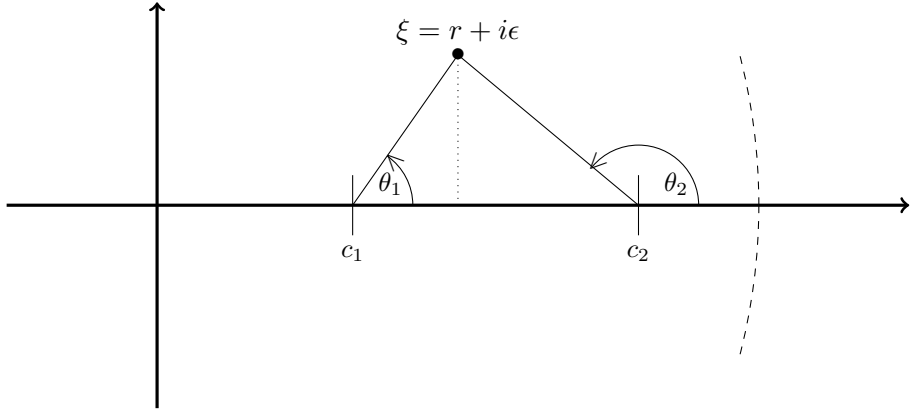
This elementary lemma will also be proved in Appendix \ref{sec:residue}.

\begin{lemma}
\label{lem:calA2penultimate}
For $x,w \in (0,\infty)$ satisfying $4x+w^2<1$, it is the case that
\begin{equation}
	\forall r \in [c_1(x),c_2(x)]\, ,\ \text{ we have }\	\pm i \cdot \widetilde{g}_{\pm}(x,r)\,
	=\, -\sqrt{-\mathcal{Q}_1(x;r)}\, ,
\end{equation}
where note that $\mathcal{Q}_1(x;r)\leq 0$ for all $r \in [c_1(x),c_2(x)]$. In particular 
\begin{equation}
\label{eq:A2formula2}
	\mathcal{A}_2(x,w)\, =\, \frac{1}{\pi}\, \int_{c_1(x)}^{c_2(x)} \frac{\sqrt{-\mathcal{Q}_1(x;r)}}{-\mathcal{Q}_2(x,w;r)}\, dr\, ,
\end{equation}
where also  $\mathcal{Q}_2(x,w;r)<0$
for $r \in [c_1(x),c_2(x)]$, since $a_1(x,r)<c_1(x)<c_2(x)<a_2(x,r)$.
\end{lemma}
We will give the details of the proof in Appendix \ref{sec:residue}. But we may indicate the key to the calculation by referring to
Figure \ref{fig:sqrtNew2}.
Recall from equation (\ref{eq:calQ1roots}) and (\ref{eq:gDefinition}) that $\widetilde{g}(x;\cdot)$ is an 
analytic continuations of $\sqrt{\mathcal{Q}_1(x,\cdot)}$. 
More precisely, we had a function $f(x;\cdot)$ defined on $\mathcal{D}_2(x)$
such that $f(x;\xi) = \sqrt{\mathcal{Q}_1(x;\xi)/\xi^2}$.
Then $\widetilde{g}(x;\cdot)$ is defined on the region $\mathcal{D}_4(x)$ and $\widetilde{g}(x;\xi) = \xi\cdot f(x;\xi)$ on $\mathcal{D}_2(x)$  (which is a subset of $\mathcal{D}_4(x)$).
It is easy to see that
\begin{equation}
	\forall r \in [c_1(x),c_2(x)]\, ,\ \text{ we have } \lim_{\epsilon \to 0^+} \left|\widetilde{g}(x;r+i\epsilon)\right|\,
	=\, 	\lim_{\epsilon \to 0^+} \left|\widetilde{g}(x;r-i\epsilon)\right|\,
	=\, \sqrt{-\mathcal{Q}_1(x;r)}\, ,
\end{equation}
because $\mathcal{Q}_1(x;r)\leq 0$ for all $r \in [c_1(x),c_2(x)]$.
So all that remain is to determine the phase, meaning $\widetilde{g}_{\pm}(x;r)/|\widetilde{g}_{\pm}(x;r)|$.
Let us call that $\Omega_{\pm}(x;r)$. Then as in Figure 5, we may see that $\Omega_{\pm}(x;r) = \pm i$, mainly because
in the product of the four square roots, exactly one has an argument which is ``approximately negative,'' being such that the phase of the argument of that square root is equal to:
$e^{i (\pi - \delta)}$ for $\xi=r+i\epsilon$,
and $e^{-i(\pi-\delta)}$ for $\xi = r-i\epsilon$, where the quantity $\delta$
is such that $\delta \to 0^+$ when $\epsilon \to 0^+$.
Taking the square-roots we get  $\Omega_{\pm}(x;r) = \pm i$.
We give more careful details in Appendix \ref{sec:residue}.

Note that the integral equation (\ref{eq:A2formula2}) is an elliptic integral:  $\mathcal{A}_2(x,w)$ can be put into the form
\begin{equation}
\label{eq:A2formula2}
\begin{split}
&\hspace{-1cm}
	\frac{1}{\pi}\cdot \int_{c_1(x)}^{c_2(x)} \frac{x\cdot \sqrt{r-c_1(x)}\, \cdot \sqrt{c_2(x)-r}\,
\cdot \sqrt{d_1(x)-r}\, \cdot \sqrt{d_2(x)-r}}{-\mathcal{Q}_2(x,w;r)}\, dr\\[3pt]
&\hspace{2cm} =\,  \int_{c_1(x)}^{c_2(x)} \frac{\mathcal{R}(x,w;r)}{\sqrt{r-c_1(x)}\, \cdot \sqrt{c_2(x)-r}\,
\cdot \sqrt{d_1(x)-r}\, \cdot \sqrt{d_2(x)-r}}\, dr\, ,
\end{split}
\end{equation}
where $\mathcal{R}(x,w;\cdot) : \C \to \C \cup \{\infty\}$ is a rational (meromorphic function), and we note that the 
denominator on the left-hand-side is the square-root of a quartic function.
But for the Legendre normal form for elliptic integrals, we prefer to use a linear fractional transformation
to transform to a quartic of the form
\begin{equation}
\label{eq:quarticNormalForm}
(1-r^2)(1-k^2 r^2)\, =\, k^2 \cdot (r-1)(r+1)(r-k^{-1})(r+k^{-1})\, ,
\end{equation}
for some number $k \in (0,1)$.
In other words, for the normal form, the roots of the quartic should come in pairs which are additive inverses.
For $\mathcal{Q}_1(x,w;r)$, the roots come in pairs which are multiplicative reciprocals (\ref{eq:Q1inversive}).

To reduce an elliptic integral to normal form one generally applies one or more
linear fractional transfomations $L(r) = (ar+b)/(cr+d)$, for example with $ad-bc=1$
and $a,b,c,d \in \R$ having inverse $L^{-1}(s) = (ds-b)/(-cs+a)$, using a substitution $r=L^{-1}(s)$.
This is a good type of transformation for such a problem.
Precomposing with $L^{-1}$ has good properties for rational functions.
Such a mapping also takes the quartic function inside the square-root to another quartic function, 
modulo multiplying by an even power of the denominator $(-cs+a)$.
It should be mentioned that $dr=dL^{-1}(s)$ also produces an even power of $(-cx+a)$.

Therefore we note, that within the family of linear fractional transformations there is one (and its inverse)
which intertwines between pairs of numbers which are multiplicative reciprocals
and pairs of numbers which are additive inverses.
It is 
\begin{equation}
\label{eq:calLdef}
	\mathcal{L}(z)\, =\, \frac{z-1}{z+1}\, ,\ \text{ with }\ \mathcal{L}^{-1}(\xi)\, =\, \frac{1+\xi}{1-\xi}\, .
\end{equation}
This satisfies $\mathcal{L}(1/z) = -\mathcal{L}(z)$ and $\mathcal{L}(-z)=1/\mathcal{L}(z)$.
After performing this transformation and rescaling, we do replace the quartic $\mathcal{Q}_1(x;r)$
by the normal form of (\ref{eq:quarticNormalForm}), with
\begin{equation}
	k^2\, =\, \big(K_1(x)\big)^2\, ,\ \text{ where }\ K_1(x)\, =\, \sqrt{1-4x^2}\, .
\end{equation}
Then the path of integration is the interval between two of the roots of (\ref{eq:quarticNormalForm}),
$1$ and $1/k$. 
We do want the path of integration to be an invterval joining two of the roots of (\ref{eq:quarticNormalForm})
to get a complete elliptic integral, which is a bit simpler than the incomplete integrals.
Unfortunately, for the Legendre normal form of an elliptic integral, it is preferred that the 
path of integration is between the pair of roots $-1$ and $1$ (instead of between the pair $1$ and $1/k$). 
So, we use another linear fractional transformation to permute the roots cyclically, while also changing $k$
from $K_1(x)$ to another value. 
The linear fractional transformation is
\begin{equation}
	\Lambda(k;z)\, =\, \frac{\left(k^{-1/2}+1\right)\left(z-k^{-1/2}\right)}
{\left(k^{-1/2}-1\right)\left(z+k^{-1/2}\right)}\, ,\ \text{ with }\
	\Lambda^{-1}(k;\xi)\, =\, k^{-1/2}\, \cdot \frac{k^{-1/2}+1+\left(k^{-1/2}-1\right)\xi}
{k^{-1/2}+1-\left(k^{-1/2}-1\right)\xi}\, ,
\end{equation}
which has the property
\begin{equation}
	\Lambda(k;k^{-1})\, =\, 1\, ,\ \Lambda(k;1)\, =\, -1\, ,\ 
	\Lambda(k;-1)\, =\, -1/\mathcal{J}(k)\ \text{ and }\
	\Lambda(k;-k^{-1})\, =\, 1/\mathcal{J}(k)\, ,
\end{equation}
for the involution $\mathcal{J} : (0,1) \to (0,1)$ given by
\begin{equation}
	\mathcal{J}(k)\, =\, (1-k^{1/2})^2/(1+k^{1/2})^2\, .
\end{equation}
It may be worthwhile to note that $\mathcal{J}$ is a decreasing function on $(0,1)$ with fixed point $(\sqrt{2}-1)^2$.

After performing the second linear fractional transformation using
$\Lambda(K_1(x);\cdot)$ and $\Lambda^{-1}(K_1(x);\cdot)$,
we do obtain another elliptic integral with the quartic of the form of (\ref{eq:quarticNormalForm})
now with
\begin{equation}
	k^2\, =\, (K_2(x))^2\, ,\ \text{ where }\ K_2(x)\, =\,\mathcal{J}(K_1(x))\,
	=\, \left(\frac{1-(1-4x^2)^{1/4}}{1+(1-4x^2)^{1/4}}\right)^2\, .
\end{equation}
And the path of integration is the interval from $-1$ to $1$, as desired.
\begin{theorem}
\label{thm:rewrite}
As usual, denote the complete elliptic integral of the first kind as 
\begin{equation}
\label{eq:FirstKind}
	K(k)\, =\, \int_0^1 \frac{1}{\sqrt{(1-t^2)(1-k^2t^2)}}\, dt\, ,
\end{equation}
and let the complete elliptic integral of the third kind be
\begin{equation}
	\Pi(k;\lambda)\, =\, \int_0^1 \frac{1}{\sqrt{(1-t^2)(1-k^2 t^2)}}\cdot \frac{1}{1-\lambda t}\, dt\, .
\end{equation}
Then we have that $\mathcal{A}_2(x,w)$ can be written as a linear combination of terms of the form
$\phi_i(x,w) \Pi(\psi_i(x,w))$ (with at most 4 terms) where each $\phi_i$ and $\psi_i$ are algebraic
expressions.
Hence the generating function $\alpha(w,z)$ is an expression involving algebraic functions
and a linear combination of several $\Pi$-functions, viewed as a vector over the field of algebraic functions
(other than the identically zero function).
\end{theorem}

We note that occasionally leading theoretical physicists make contributions to algebra and number theory
by determining certain physical quantities in terms of algebraic or number theoretic expressions.
For example, Boos, Korepin, Nishiyama and Shiroishi found that certain correlation functions of the XXZ ground
state (for the quantum spin chain) are expressible in terms of rational linear combinations of the products of
$\ln(2)$ and factors of the Riemann zeta functions $\zeta$ evaluated at the first $n$ odd positive integers
\cite{KorepinEtAl}.
Our main theorem may be viewed as following their model.

The following is an {important remark.}
\begin{remark}
\label{rem:PinskySRW}
In \cite{Pinsky1}, Ross Pinsky proved the combinatorial fact
\begin{equation}
\label{eq:Pinsky2}
	\frac{A(N,j)}{\binom{2N}{N}\binom{2N}{N}}\, =\, 
	\sum_{\substack{t_1,\dots,t_j \in \Z\\0\leq t_1\leq\dots\leq t_j\leq 2N}} 
\mathbf{P}\big( U_{t_1}=\dots=U_{t_j}=0\, \big|\, U_{2N}=V_{2N}=0\big)\, ,
\end{equation}
where $(U_t,V_t)$ performs the standard $d=2$ dimensional SRW starting from $(0,0)$.
See his equation 4.2 in the first part of his proof of his Lemma 2.
\end{remark}

The standard $d=2$ dimensional SRW is such that $(U_{t+1}-U_t,V_{t+1}-V_t)$ are iid increments equally likely
to be any of the four elements from $\{(1,0),(0,1),(-1,0),(0,-1)\}$.
The transformation that Pinsky gives in his paper is $(X_t,Y_t) = (U_t+V_t,U_t-V_t)$,
so that $X_t$ and $Y_t$ independently perform $d=1$ dimensional SRWs.
Then it is seen that $\mathbf{P}(U_{2N}=V_{2N}=0)=2^{-4N} \binom{2N}{N}\binom{2N}{N}$. Therefore, Pinsky's formula (\ref{eq:Pinsky2}) can be rewritten as
\begin{equation}
\label{eq:PrePolya}
	2^{-4N} A(N,j)\, =\, \sum_{\substack{t_1,\dots,t_j \in \Z\\0<t_1<\dots<t_j<2N}} 
\mathbf{P}\big( U_{t_1}=\dots=U_{t_j}=0\ \text{ and }\ U_{2N}=V_{2N}=0\big)\, ,
\end{equation}
This is also related to P\'olya's $d=2$ dimensional result.
We point to McKean's textbook, in which he discusses
P\'olya's result in Section 4.3, \cite{McKean}. 
That is more accessible than P\'olya's original \cite{Polya}.
The precise formula of P\'olya is the following
\begin{equation}
	\sum_{N=0}^{\infty} \mathbf{P}(U_{2N}=V_{2N}=0) z^N\, =\, \frac{1}{\pi}\, K(z)\, ,
\end{equation}
for $|z|<1$, where $K(z)$ is as in (\ref{eq:FirstKind}).
Let us perform a consistency check for our main formulas by 
noting that when $w=0$ we have $\mathcal{A}_1(x,w)=\mathcal{A}_1(x,0)=0$ due to formula 
(\ref{eq:A2formula2}).
But in Lemma \ref{lem:calA2penultimate}, if we send $w \to 0$, we have 
\begin{equation}
-\mathcal{Q}_2(x,w;r)\, =\, 
-\mathcal{Q}_2(x,0;r)\, =\, -\mathcal{Q}_1(x;r)\, .
\end{equation}
But then, by Lemma \ref{lem:residue1}, we have
\begin{equation}
	\alpha(0,x^2)\, =\, 0 + \mathcal{A}_2(x,0)\, =\, \frac{1}{ \pi x}\, \int_{c_1(x)}^{c_2(x)} 
\frac{1}{\sqrt{\big(r-c_1(x)\big)\big(c_2(x)-r\big)\big(\frac{1}{c_1(x)}-r\big)\big(\frac{1}{c_2(x)}-r\big)}}\, dr\, .
\end{equation}
Using the methods of Appendix \ref{sec:Elliptic}, this can be seen to give the desired specialization $\alpha(x,w=0)=K(x)/\pi$.

\section{Summary and outlook}

We have presented some explicit formula for the generating function of the combinatorial array $A(N,j)$.
This arrray was introduced by Ross Pinsky in his beautiful combinatorial formula for the 2nd moment
of $Z_{n,k}$ from the generalized Ulam problem.
Pinsky then obtained rigorous upper and lower bounds, sufficient for his purpose of determining when the sequence $Z_{n,{k_n}}$
satisfies a weak law of large numbers or does not satisfy such, depending on the behavior of the sequence $k_n$
as $n \to \infty$.
But for  reasons discussed in Appendix \ref{sec:SK} and in the introduction, we believe
that more precise asymptotics are desirable.

We presented several formulas for the generating function, but not for the precise asymptotics. Using the Chebyshev
inequality applied to 
\begin{equation}
\alpha(w,x^2)\, =\, \sum_{j,N=0}^{\infty} w^j x^{2N} A(N,j)\, ,
\end{equation}
since all the combinatorial terms $A(N,j)$ are nonnegative, we may obtain an upper bound
\begin{equation}
	A(N,j)\, \leq\, \min_{w>0} \min_{z>0} \frac{\alpha(w,x^2)}{w^j x^{2N}}\, =\,
\min_{w>0} \min_{z>0} \frac{\mathcal{A}_1(x,w)+\mathcal{A}_2(x,w)}{w^j x^{2N}}\, .
\end{equation}
We have not yet carried out this optimization problem.
We mention that the collaboration resulting in the present paper took place after one of the authors
worked on his own (with mentorship) for  precisely one summer, and then both coauthors collaborated
on the present paper after that work was carried out.
This paper represents the end product of a certain collaboration. 
Even if further results seem like they should naturally be contained herein, we have limited this paper to 
the products of that research experience.

So the above calculation represents our primary interest in future research in this area.
As we said before, a next natural question is to calculate $\mathbf{E}[Z_{n,k} Z_{n,\ell}]$ for $k$ and $\ell$
being allowed to differ.
We expect that is an important first step towards calculating all higher moments by iterating Pinsky's
procedure for the 2nd moment.
But as with the problem above, we have not carried that out yet.
Either we will return to this problem, possibly in collaboration with other researchers.
Or else perhaps some other researchers will resolve this, indepenent from ourselves.

\label{sec:Outlook}

\section*{Acknowledgments}

S.S.~is grateful to Dongsheng Wu for a  helpful conversation. The work of S.S.~was partly supported by 
a Simons collaboration grant.

\appendix 

\vspace{1cm}

\noindent
{\LARGE \bf Appendices}

\small

\section{Bonferroni inequalities and Bell polynomials}
\label{sec:Bon}

The principle of inclusion-exclusion (PIE) and the associated inequalities related to the PIE
are a well-known topic at the intersection of combinatorics and probability theory.
We will follow the exposition from the textbook of Charalambides \cite{Charalambides}
which we appreciate, greatly. The descriptions there are clear, and also sufficiently complete
to be useful to newcomers. For us, it is nice that he discusses the Bonferroni inequalities
and the Bell polynomials in some detail.

Recall that we are interested in the number of cardinality $k$ subsequences of $\{1,\dots,n\}$
which are increasing for a uniform random permutation $\pi \in S_n$. Let us denote by $C_{n,k}$
the collection of $k$-tuples $\boldsymbol{i} = (i_1,\dots,i_k) \in \Z^k$ such that
\begin{equation}
	1\, \leq\, i_1\, <\, i_2\, <\, \dots\, <\, i_k\, \leq\, n\, .
\end{equation}
So, for instance, the cardinality of this set is 
\begin{equation}
	|C_{n,k}|\, =\, \binom{n}{k}\, .
\end{equation}
Given a permutation $\pi \in S_n$ and $(i_1,\dots,i_k) \in C_{n,k}$, let us denote the 
indicator for $(i_1,\dots,i_k)$ being an increasing subsequence of $\pi$:
\begin{equation}
\label{eq:calPdef}
	\mathcal{P}_{n,k}(\pi;(i_1,\dots,i_k))\, =\, 
	\prod_{a=1}^{k-1} \mathbf{1}_{(0,\infty)}(\pi_{i_{a+1}} - \pi_{i_a})\, .
\end{equation}
So, for example, we know that
\begin{equation}
	\forall (i_1,\dots,i_k) \in C_{n,k}\, ,\ \text{ we have }\ 
	\sum_{\pi \in S_n} 	\mathcal{P}_{n,k}(\pi;(i_1,\dots,i_k))\, =\, 
	\frac{n!}{k!}\, ,
\end{equation}
by the discussion around equations (\ref{eq:jIncreasing}), (\ref{eq:jpiSubset})
and (\ref{eq:jpiOrdered}).
Now let $\widetilde{C}^{(r)}_{n,k}$ denote the collection of all sets 
\begin{equation}
	\left\{\boldsymbol{i}^{(1)},\dots,\boldsymbol{i}^{(r)}\right\}\, ,
\end{equation}
satisfying
\begin{itemize}
\item[$\bullet$] for each $s \in \{1,\dots,r\}$ we have $\boldsymbol{i}^{(s)} = (i^{(s)}_1,\dots,i^{(s)}_k) \in C_{n,k}$,
\item[$\bullet$] for each $s,t \in \{1,\dots,r\}$, if $s\neq t$ then there is some $a \in \{1,\dots,k\}$
such that $i^{(s)}_a \neq i^{(t)}_a$.
\end{itemize}
Therefore, we have, for instance, that the cardinality is equal to
\begin{equation}
	\left|\widetilde{C}^{(r)}_{n,k}\right|\, =\, \binom{\binom{n}{k}}{r}\, 
	=\, \frac{N(N-1)\cdots (N-r+1)}{r!} \Bigg|_{N=\binom{n}{k}}\, .
\end{equation}
Now, following the notation from Chapter 4 of Charalambides, let us define
\begin{equation}
\label{eq:calS}
	\mathcal{S}_{n,k}^{(r)}\, =\, \sum_{\pi \in S_n}\, \sum_{\{\boldsymbol{i}^{(1)},\dots,\boldsymbol{i}^{(r)}\} \in
	\widetilde{C}^{(r)}_{n,k}}\,  \prod_{s=1}^{r} \mathcal{P}_{n,k}(\pi;\boldsymbol{i}^{(s)})\, .
\end{equation}
This is equal to the sum over all 
$(\boldsymbol{i}^{(1)},\dots,\boldsymbol{i}^{(r)}) \in \widetilde{C}^{(r)}_{n,k}$
of the number of permutations in $S_n$ such that each of $\boldsymbol{i}^{(1)}$, \dots,
$\boldsymbol{i}^{(r)}$ is an increasing subsequence for $\pi$.
So, for example, 
\begin{equation}
	\frac{1}{n!}\, \mathcal{S}_{n,k}^{(1)}\, =\, \binom{n}{k} \cdot \frac{1}{k!}\, ,
\end{equation}
as in equation (\ref{eq:Zmean}). Note that we are following the notation of the {\em combinatorics}
textbook of Charalambides, so that we focus on cardinalities, instead of expectations and probabilities.
(That is what we needed to noramlize by $1/n!$ on the left-hand-side of the equation above.)
Now, following the notation from Charalambides, let us define
\begin{equation}
	\mathcal{L}^{(r)}_{n,k}\, =\, \left|\left\{ \pi \in S_n\, :\, 
	\exists \{\boldsymbol{i}^{(1)},\dots,\boldsymbol{i}^{(r)}\} \in \widetilde{C}^{(r)}_{n,k}\, ,\ 
	\text{ s.t. } \ \prod_{s=1}^{r} \mathcal{P}_{n,k}(\pi,\boldsymbol{i}^{(s)})=1
\right\}\right|\, .
\end{equation}
This is equal to the number of permutations $\pi \in S_n$ such that there are at least $r$
distinct subsequences $\boldsymbol{i}^{(1)},\dots,\boldsymbol{i}^{(r)} \in C_{n,k}$,
each of cardinality $k$, such that all of the $r$ subsequences are increasing for $\pi$.
Then note that we have, using the definition of $Z_{n,k}(\pi)$ from (\ref{eq:defZnk}),
\begin{equation}
	\frac{1}{n!}\, \mathcal{L}^{(r)}_{n,k}\,
	=\, \mathbf{P}(\{\pi \in S_n\, :\, Z_{n,k}(\pi) \geq r\})\, .
\end{equation}
In principle, we are most interested in $\mathcal{L}^{(1)}_{n,k}$ because this is the cardinality which 
arises in the formula for the length of the longest increasing subsequence:
\begin{equation}
\label{eq:ApplicationL1}
	\frac{1}{n!}\, \mathcal{L}^{(1)}_{n,k}\, 
	=\,  \mathbf{P}(\{\pi \in S_n\, :\, Z_{n,k}(\pi) \geq 1\})\,
	=\, \mathbb{P}(\{\pi \in S_n\, :\, L_n(\pi) \geq k\})\, .
\end{equation}
But Charalambides states the Bonferroni inequalities for general $r$. Therefore, we defined
the terms for general $r$, in order to most closely match his notation. (Note that that variable which he calls $k$ is what
we call $r$, and there are other minor changes {\em mutatis mutandis}.)
\begin{proposition}[Version of the Bonferroni inequalities]
Suppose we have $n \in \N$ and we have chosen $k \in \{1,\dots,n\}$.
Then, for each $r \in \{1,\dots,\binom{n}{k}\}$ and for each
$R \in \{0,1,\dots,|\widetilde{C}^{(r)}_{n,k}|\}$ 
we have
\begin{equation}
	(-1)^{R-r+1} \left(\mathcal{L}_{n,k}^{(r)} - \sum_{s=r}^R (-1)^{s-r} \binom{s-1}{r-1} \mathcal{S}_{n,k}^{(s)} \right)\, \geq\, 0\, .
\end{equation}
\end{proposition}
We will not provide the proof of this proposition, since we have copied it from \cite{Charalambides}
where it is Theorem 4.5, equation (4.23).
Our main interest is in the case of $r=1$. So, summarizing that case, the proposition implies
that for each $R \in \{1,\dots,\binom{n}{k}\}$, we have:
\begin{itemize}
\item[$\bullet$] if $R$ is even then
\begin{equation}
	\mathcal{L}_{n,k}^{(1)}\, \geq\, \sum_{s=1}^{R} (-1)^{s-1} \mathcal{S}_{n,k}^{(s)}\, ;
\end{equation}
\item[$\bullet$] whereas, if $R$ is odd then
\begin{equation}
	\mathcal{L}_{n,k}^{(1)}\, \leq\, \sum_{s=1}^{R} (-1)^{s-1} \mathcal{S}_{n,k}^{(s)}\, .
\end{equation}
\end{itemize}
Of course, PIE also gives (as in Theorem 4.2 of Charalambides for example)
\begin{equation}
	\mathcal{L}_{n,k}^{(1)}\, =\, \sum_{s=1}^{|C_{n,k}|} (-1)^{s-1} \mathcal{S}_{n,k}^{(s)}\, .
\end{equation}
So, if one could calculate $\mathcal{S}_{n,k}^{(s)}$ for all $s$ up to $|C_{n,k}|=\binom{n}{k}$,
then this gives a method of calculating $\mathcal{L}_{n,k}^{(1)}$ from them, in principle.
But the goal of the Bonferroni inequalities is to hopefully obtain
inequalities that are useful without requiring all $\mathcal{S}_{n,k}^{(s)}$ for $s$ going up
to $\binom{n}{k}$, but merely those for $s \in \{1,\dots,R\}$.
It is hoped that for some $R\leq \binom{n}{k}$ this gives 
an easier method of bracketing $\mathcal{L}_{n,k}^{(1)}$ in order to calculate
$\mathbb{P}(L_n \geq k)$.

The next step is to simplify the calculation of the quantities $\mathcal{S}_{n,k}^{(r)}$.
What Pinsky's method allow us to calculate is moments of $Z_{n,k}$.
One may note that $Z_{n,k}(\pi)$ is the total number of cardinality $k$ subsequences, $\boldsymbol{i} \in C_{n,k}$
such that $\boldsymbol{i}$ is increasing for $\pi$.
Then it is easy to see that 
\begin{equation}
	 \sum_{\{\boldsymbol{i}^{(1)},\dots,\boldsymbol{i}^{(r)}\} \in
	\widetilde{C}^{(r)}_{n,k}}\,  \prod_{s=1}^{r} \mathcal{P}_{n,k}(\pi;\boldsymbol{i}^{(s)})\, 
	=\, \binom{Z_{n,k}(\pi)}{r}\, =\, \frac{1}{r!}\, Z_{n,k}(\pi) \big(Z_{n,k}(\pi)-1\big) \cdots \big(Z_{n,k}-r+1\big)\, .
\end{equation}
In other words, combining the earlier calculations, we have for $R \in \{1,\dots,\binom{n}{k}\}$:
\begin{itemize}
	\item if $R$ is even, then 
\begin{equation}
	\mathbf{P}(L_n \geq k)\, =\, \mathbf{P}(Z_{n,k}\geq 1)\, \geq\, \sum_{s=1}^{R} (-1)^{s-1} \mathbf{E}\left[\binom{Z_{n,k}}{s}\right]\, ;
\end{equation}
	\item while if $R$ is odd, then
\begin{equation}
	\mathbf{P}(L_n \geq k)\, =\, \mathbf{P}(Z_{n,k}\geq 1)\, \leq\, \sum_{s=1}^{R} (-1)^{s-1} \mathbf{E}\left[\binom{Z_{n,k}}{s}\right]\, .
\end{equation}
\end{itemize}
But more generally, for each $r \in \{1,\dots,\binom{n}{k}\}$ and for each
$R \in \{0,1,\dots,|\widetilde{C}^{(r)}_{n,k}|\}$ 
we have
\begin{equation}
	(-1)^{R-r+1} \left(\mathbf{P}(Z_{n,k}\geq r) - \sum_{s=r}^R (-1)^{s-r} \binom{s-1}{r-1}  \mathbf{E}\left[\binom{Z_{n,k}}{s}\right] \right)\, \geq\, 0\, .
\end{equation}
\begin{remark}
In an earlier version of this article, we promoted the idea of the Bell polynomials, not realizing the simple connection to the factorial moments.
\end{remark}
Precisely, the important fact is that, defining exponential Bell partition polynomials as $B_0\equiv 0$ and for $r \in \N = \{1,2,\dots\}$
\begin{equation}	
\label{eq:BellDef}
	B_r(x_1,\dots,x_r)\, =\, \sum_{s=0}^{r-1} \binom{r+1}{s} x_{s+1} B_{r-1-s}(x_1,\dots,x_{r-1-s})\, ,
\end{equation}
we have that the following 1-variable specialization gives the identity 
\begin{equation}
	\forall r \in \{0,1,\dots\}\, ,\ \text{ we have }\ B_r(-0!z,-1!z,\dots,-(r-1)!z)\, =\, \frac{1}{r!}\, (z)_{r}\, ,
\end{equation}
where $(z)_r$ is the Pochhammer symbol, falling factorial from (\ref{eq:PochhammerFallingFactorial}).
A reference for the Bell polynomials is Charalambides, Theorem 11.2 \cite{Charalambides}.

Even though the necessity is now bypassed, let us describe our old observations that we largely 
transcribed from Charalambides.
The reason that the Bell polynomials are relevant is that according to Newton's identity
\begin{equation*}
	\text{\rm\url{https://en.wikipedia.org/wiki/Newton's_identities}}
\end{equation*}
we have the following relation between symmetric functions
\begin{equation}
	r\cdot e_r(x_1,\dots,x_N)\, =\, \sum_{s=1}^{r} (-1)^{s-1} e_{r-s}(x_1,\dots,x_N) p_s(x_1,\dots,x_N)\, ,
\end{equation}
where the power sum symmetric functions are
\begin{equation}
	p_r(x_1,\dots,x_N)\, =\, \sum_{\kappa=1}^{N} x_{\kappa}^r\, ,
\end{equation}
and the elementary symmetric functions are 
\begin{equation}
	e_r(x_1,\dots,x_N)\, =\, \sum_{i_1,\dots,i_r \in \{1,\dots,N\}} 
	\left(\prod_{s=1}^{r-1} \mathbf{1}_{(0,\infty)}(i_{s+1}-i_s)\right) x_{i_1} x_{i_2} \cdots x_{i_r}\, .
\end{equation}
Relatedly, there are identities for certain types of generating functions of these symmetric functions:
\begin{equation}
\label{eq:NewtonGenFun}
	1+\sum_{r=1}^{N} (-t)^k e_r(x_1,\dots,x_N)\, =\, \exp\left(-\sum_{r=1}^{\infty} \frac{t^r}{r} p_r(x_1,\dots,x_N)\right)\, .
\end{equation}
This is relevant because according to (\ref{eq:calS})
\begin{equation}
\label{eq:calS2}
	\mathcal{S}_{n,k}^{(r)}\, =\, \sum_{\pi\in S_n} e_r(\mathfrak{X}_1(\pi),\dots,\mathfrak{X}_N(\pi))\, ,
\end{equation}
where  we enumerate the elements of $C_{n,k}$
as
\begin{equation}
	C_{n,k}\, =\, \left\{\boldsymbol{i}^{(1)},\dots,\boldsymbol{i}^{(N)}\right\}\, ,
\end{equation}
and, for each $\kappa \in \{1,\dots,N\}$, we define
\begin{equation}
	\mathfrak{X}_{\kappa} : S_n \to \{0,1\}\, ,\ \text{ defined by the formula }\
	\mathfrak{X}_{\kappa}(\pi)\, =\, \mathcal{P}_{n,k}\left(\pi;\boldsymbol{i}^{(\kappa)}\right)\, .
\end{equation}
where $\mathcal{P}_{n,k}(\cdot;\cdot)$ is as in (\ref{eq:calPdef}).
But then, since each $\mathfrak{X}_\kappa$ is an indicator function, taking values in 
$\{0,1\}$,
\begin{equation}
\label{eq:ZPower}
	\forall r \in \N=\{1,2,\dots\}\, ,\ \text{ we have }\ 
	p_r(\mathfrak{X}_1,\dots,\mathfrak{X}_N)\, =\, Z_{n,k}\, =\, \sum_{s=1}^{N} \mathfrak{X}_{\kappa}\, .
\end{equation}
If we did not have this simple formula, then a formula for each $r \in \N$ would be useful more generally:
\begin{equation}
\label{eq:Wikipedia}
	e_r(x_1,\dots,x_N)\, =\, \frac{(-1)^r}{r!}\, B_r(-0!p_1(x_1,\dots,x_N),\dots,-(r-1)!p_r(x_1,\dots,x_r))\, ,
\end{equation}
which essentially restates (\ref{eq:NewtonGenFun}) via (\ref{eq:BellDef}) (which is equivalent to 
the generating function identity stating that
$\sum_{r=1}^{\infty} B_n(x_1,\dots,x_n) t^n/n!$ is 
equal to $\exp\left(\sum_{r=1}^{\infty} x_r t^r/r!\right)$ as in Charalambides Corollary 11.1).
See, for example, Example 11.5 on page 427 in \cite{Charalambides}.

\section{Review of the physicists' replica approach to the SK model}
\label{sec:SK}

Suppose that for each pair $i,j \in \N = \{1,2,\dots\}$ there is a standard normal random variable $\mathsf{J}_{i,j}$,
such that all the $\mathsf{J}_{i,j}$'s are independent.
Then, for each $N \in \N$, the $N$-spin Sherrington-Kirkpatrick Hamiltonian is a random function
\begin{equation}
\mathsf{H}_N : \{+1,-1\}^N \to \R\, ,\ \text{ by the formula }\
\mathsf{H}_N(\boldsymbol{\sigma})\, =\, -\sum_{i=1}^{N} \sum_{j=1}^{N} \frac{\mathsf{J}_{i,j}}{\sqrt{2N}}\, \sigma_i \sigma_j\, ,
\end{equation}
for each $\boldsymbol{\sigma} = (\sigma_1,\dots,\sigma_N) \in \{+1,-1\}^N$.
Originally, the simplest question that phyicists sought to answer was the formula for the limiting free energy.

The finite-size approximation to the thermodynamic free energy is $f_N : (0,\infty) \to \R$, given by the formula
\begin{equation}
\mathsf{f}_N(\beta)\, =\, -\frac{1}{N \beta}\, \ln\left(\mathsf{Z}_N(\beta)\right)\, ,\ \text{ where }\ \mathsf{Z}_N(\beta)\, =\, \sum_{\sigma_1,\dots,\sigma_N \in \{+1,-1\}}
\exp\big(-\beta \mathsf{H}_N(\boldsymbol{\sigma})\big)\, .
\end{equation}
The original interest was in the limit of the quenched free energy
\begin{equation}
f^Q(\beta)\, =\, \lim_{N \to \infty} \mathbf{E}[\mathsf{f}_N(\beta)]\, ,
\end{equation}
as well as a proof argument for the fact that $\mathsf{f}_N(\beta)$ satisfies a weak law of large numbers in the $N \to \infty$ limit.

One approach to trying to guess the formula for $f^Q(\beta)$ is to actually calculate positive integer moments of the 
random function $\mathsf{Z}_N(\beta)$, which is called the random partition function for $N$ spins.
An excellent reference for mathematicians interested in one perspective on the physicists' ``replica trick,'' is the 
old review article of van Hemmen and Palmer \cite{vanHemmenPalmer}.
Following their notation, let us define, for each $n \in \N$, the function $\varphi_{N}^{(n)}:(0,\infty) \to \R$ given by the formula
\begin{equation}
\label{eq:momentsPF}
\varphi_{N}^{(n)}(\beta)\, =\, \frac{1}{N}\, \ln\left(\mathbf{E}\left[\big(\mathsf{Z}_N(\beta)\big)^n\right]\right)\, .
\end{equation}
Note that one could extend the definition to $\varphi_N^{(\nu)}(\beta)$ for all $\nu \in (0,\infty)$. Then one would have
$\mathbf{E}[-\beta \mathsf{f}_N(\beta)]$ is equal to $\lim_{\nu \to 0^+} (\varphi_N^{(\nu)}(\beta)-1)/\nu$, by properties such as convexity,
which insures existence of 1-sided derivatives.

The real ``trick'' in the physicists' method is that they extrapolate from integer $n>1$ to the limit $\nu \to 0^+$, in a non-trivial
way. Because the $\mathsf{J}_{i,j}$'s are all IID standard, normal random variables, we see that the $\mathsf{H}_N(\boldsymbol{\sigma})$'s 
(for varying $\boldsymbol{\sigma}$'s)
are jointly distributed (centered) Gaussian normal random variables, with covariance
\begin{equation}
\mathbf{E}\left[\mathsf{H}_N(\boldsymbol{\sigma})\mathsf{H}_N(\boldsymbol{\tau})\right]\, =\, \frac{1}{2N}\, \sum_{i=1}^{N} \sum_{j=1}^{N}
\sigma_i \sigma_j \tau_i \tau_j\, =\, \frac{N}{2}\, \big(q_N(\boldsymbol{\sigma},\boldsymbol{\tau})\big)^2\, ,\
\text{ where }\
q_N(\boldsymbol{\sigma},\boldsymbol{\tau})\, =\, \frac{1}{N}\, \sum_{i=1}^{N} \sigma_i \tau_i\, . 
\end{equation}
Using this, and the fact that
\begin{equation}
\big(\mathsf{Z}_N(\beta)\big)^n\, =\, \sum_{\boldsymbol{\sigma}^{(1)},\dots,\boldsymbol{\sigma}^{(n)} \in \{+1,-1\}^N} 
\exp\left(-\beta \sum_{r=1}^{n} \mathsf{H}_N\big(\boldsymbol{\sigma}^{(r)}\big)\right)\, ,
\end{equation}
the usual formula for the moment generating function of a Gaussian random variable gives
\begin{equation}
\label{eq:LargeDeviation}
\mathbf{E}\left[\big(\mathsf{Z}_N(\beta)\big)^n\right]\,
=\, \sum_{\boldsymbol{\sigma}^{(1)},\dots,\boldsymbol{\sigma}^{(n)} \in \{+1,-1\}^N}  \exp\left(\frac{\beta^2 N}{4}\, \sum_{r=1}^{n} \sum_{s=1}^{n} \left(q_N\big(\boldsymbol{\sigma}^{(r)},\boldsymbol{\sigma}^{(s)}\big)\right)^2\right)
\end{equation}
Therefore, the $\varphi_N^{(n)}(\beta)$ functions have thermodynamic limits
\begin{equation}
\varphi^{(n)}(\beta)\, =\, \lim_{N \to \infty} \varphi^{(n)}_N(\beta)\, ,
\end{equation}
which are calculated by large-deviation theory techniques. Indeed van Hemmen and Palmer 
defined a matrix-indexed Hamiltonian, for any $Q = (Q_{r,s})_{r,s=1}^{n} \in [-1,1]^{n\times n}$,
\begin{equation}
\mathcal{H}_Q : \{+1,-1\}^n \to \R\, ,\ \text{ by }\ \mathcal{H}_Q(\vec{S})\, =\, -\frac{1}{2}\, \sum_{r=1}^{N} \sum_{s=1}^{N}
\Big(Q_{r,s} S_r S_s - \frac{1}{2} Q_{r,s}^2\Big)\, ,
\end{equation}
for $\vec{S} = (S_1,\dots,S_n) \in \{+1,-1\}^n$. Then they argued that the formula for $-\varphi^{(n)}(\beta)/\beta$ is equal to 
the minimum of 
\begin{equation}
\label{eq:PFmin}
\Phi^{(n)}(Q)\, \stackrel{\mathrm{def}}{:=}\,
-\frac{1}{\beta} \ln\left( \sum_{S_1,\dots,S_n \in \{+1,-1\}} \exp\left(-\beta^2 \mathcal{H}_Q(\vec{S})\right)\right)\, ,
\end{equation}
varying over all matrices $Q \in [-1,1]^{n\times n}$.
The matrix $Q$ is akin to a Lagrange multiplier for the optimization problem associated to the large deviation type problem in (\ref{eq:LargeDeviation}).

There are various ways of arriving at this framework, some of which do not seem equivalent until extra work is done.
We have presented van Hemmen and Palmer's particular approach, which is somewhat different than the original
approach of Sherrington and Kirkpatrick \cite{SherringtonKirkpatrick}.
In \cite{vanHemmenPalmer}, they comment on the relation between the two works.
In their paper, they show that the arg-min is of the form
\begin{equation}
Q\, =\, \begin{bmatrix} 1 & q_0 & q_0 & \dots & q_0  \\ q_0 & 1 & q_0 & \dots & q_0  \\ q_0 & q_0  & 1 & \dots & q_0 \\
\vdots & \vdots & \vdots & \ddots & \vdots \\
q_0 & q_0 & q_0 & \dots & 1 \end{bmatrix}\, ,
\end{equation}
for some number $q_0$ which could then be calculated in terms of a self-consistency formula or in terms of certain Gaussian integrals (against
hyperbolic trigonometric functions).

A good reference for the physicists' perspective on replica symmetry breaking is the monograph of Mezard, Parisi and Virasoro \cite{MPV}.
In Chapter 3 they consider the analogous problem to the presentation above, but with an eye towards ``extrapolating'' from $n \in \N$ to $\nu \in (0,1)$.
This means guessing, but they must make a good educated guess to get something that is not senseless. Even Sherrington and Kirkpatrick noted in their
original paper the unphysical property that the entropy becomes negative below the critical temperature, using their approach.
Instead Parisi considered a sequence of ultrametric generalizations of the $Q$-matrix starting with the structure
\begin{equation}
	Q\, =\, \left[ \begin{array}{c|c}
\begin{matrix} 1 & q_1 & \dots & q_1 \\ q_1 & 1 & \dots & q_1 \\ \vdots & \vdots & \ddots & \vdots \\ q_1 & q_1 & \dots & 1\end{matrix} 
&
\begin{matrix} q_0 \end{matrix}\\
\hline
\begin{matrix} q_0 \end{matrix}
&
\begin{matrix} 1 & q_1 & \dots & q_1 \\ q_1 & 1 & \dots & q_1 \\ \vdots & \vdots & \ddots & \vdots \\ q_1 & q_1 & \dots & 1\end{matrix} 
\end{array}
\right]\, ,
\end{equation}
for a pair of numbers $0\leq q_0<q_1\leq 1$. 
In the matrix, the off-diagonal blocks, with just a $q_0$ are meant to indicate blocks of the matrix wherein every matrix entry
is equal to $q_0$. Similar to this, in the diagonal blocks, every matrix is equal to $q_1$ except for the diagonal  matrix entries which are $Q_{i,i}=1$ for each $i \in \{1,\dots,n\}$.

This would be considered 1 level of replica symmetry breaking.
Because of the role of the parameter $\nu-1$, one seeks the maximum
of the analogue of equation (\ref{eq:PFmin}) when $\nu \in (0,1)$, instead of the maximum.
The motivation for this particular choice is beyond our ability to review. We do recommend consulting the textbook \cite{MPV}.
An example of a $Q$ matrix with 2 levels of replica symmetry breaking is shown in Figure \ref{fig:2RSB}.
\subsection{Mathematical physics and probability developments}
Years later, a major step was taken in describing this approach by Bernard Derrida, as in  \cite{Derrida1} and \cite{Derrida2}, using the Random Energy Model,
which is an example of a problem related to {\em extreme value statistics}, as arises in the Gumbel distribution, for example.
Then Aizenman, Lebowitz and Ruelle \cite{ALR} as well as Ruelle \cite{Ruelle}, gave further perspectives, including by bringing in the appropriate
Poisson point process for Derrida's Generalized Random Energy Model. We also recommend the article of Aizenman and Ruzmaikina \cite{AizenmanRuzmaikina}.

Eventually the conjecture of Parisi was proved. First, Guerra and Toninelli invented a method of interpolation that brought easy new results \cite{GuerraToninelli}.
This was followed by an important paper of Guerra \cite{Guerra}.
Aizenman, Sims and one of the authors contributed a short paper \cite{ASS}.
Then Michel Talagrand solved the problem in a tour de force \cite{Talagrand}.
Later, Arguin and Aizenman  showed a simple proof if one assumes only a finite level of replica symmetry
breaking \cite{ArguinAizenman}.
Eventually, Dmitry Panchenko proved that the Parisi ansatz can be proved as a consequence
of just the Ghirlanda-Guerra identities, related to some results also arising from the Aizenman-Contucci identities \cite{Panchenko}
An excellent place to learn all this is the textbook by Panchenko \cite{PanchenkoTextbook}.
Another relevant reference is by Auffinger and Chen \cite{AuffingerChen}
wherein the stochastic optimization problem consequent to Parisi's ansatz is analyzed clearly and carefully.

\begin{figure}
\begin{center}
{\footnotesize
$$
	Q\, =\, \left[ \begin{array}{c|c}
\begin{array}{c|c}
\begin{matrix} 1 & q_2 & \dots & q_2 \\ q_2 & 1 & \dots & q_2 \\ \vdots & \vdots & \ddots & \vdots \\ q_2 & q_2 & \dots & 1\end{matrix} 
&
\begin{matrix} q_1 \end{matrix}\\
\hline
\begin{matrix} q_1\end{matrix}
&
\begin{matrix} 1 & q_2 & \dots & q_2 \\ q_2 & 1 & \dots & q_2 \\ \vdots & \vdots & \ddots & \vdots \\ q_2 & q_2 & \dots & 1\end{matrix} 
\end{array} & 
\begin{matrix} q_0 \end{matrix}\\
\hline
\begin{matrix} q_0 \end{matrix} &
\begin{array}{c|c}
\begin{matrix} 1 & q_2 & \dots & q_2 \\ q_2 & 1 & \dots & q_2 \\ \vdots & \vdots & \ddots & \vdots \\ q_2 & q_2 & \dots & 1\end{matrix} 
&
\begin{matrix} q_1 \end{matrix}\\
\hline
\begin{matrix} q_1\end{matrix}
&
\begin{matrix} 1 & q_2 & \dots & q_2 \\ q_2 & 1 & \dots & q_2 \\ \vdots & \vdots & \ddots & \vdots \\ q_2 & q_2 & \dots & 1\end{matrix} 
\end{array}
\end{array}
\right]
$$
}
\caption{
\label{fig:2RSB}
This is a schematic of the ultrametic block-matrix structure for a $Q$ matrix with 2 levels of replica symmetry
breaking (or 2 RSB). This type of matrix
is parametrized by the numbers
$0\leq q_0<q_1<q_2 \leq 1$.
}
\end{center}
\end{figure}

The most important point for us, here, in the present paper, is that the first step towards the full understanding of the Sherrington-Kirkpatrick
model using the physicists' approach was an exact, simple formula for the higher moments of the random partition function as in (\ref{eq:momentsPF}).

The hope of bringing the physicists' methods to bear on the Ulam problem motivates a more careful
analysis of the higher moments in the generalized Ulam problem, as we begin to do in the present article.

\section{Roots of a class of quadratic equations}
\label{sec:roots}
The elementary results of this appendix are used several times throughout the calculations
in subsequent sections. So we gather the notation, here.
For a quadratic equation of the form
\begin{equation}
\label{eq:cquad}
	\xi^2+1 -  \kappa \xi\, =\, 0\, , 
\end{equation}
where $\kappa \in \R$,
the quadratic formula for the roots gives
\begin{equation}
	\xi_{\pm}\, =\, \frac{\kappa \pm \sqrt{\kappa^2-4}}{2}\, ,
\end{equation}
which in turn can be written as 
\begin{equation}
\label{eq:csqrt}
	\xi_{\pm}\, =\, \frac{1}{4}\left(\sqrt{\kappa+2} \pm \sqrt{\kappa-2}\right)^2\, ,
\end{equation}
if we know that $\kappa\geq 2$.
See Remark \ref{rem:sqrt}.
Let us denote $\rho_+$ and $\rho_-$ as 
\begin{equation}
	\rho_{\pm}\, =\, \sqrt{\kappa \pm 2}\, .
\end{equation}
Therefore, 
\begin{equation}
\label{eq:xiRoots}
	\xi_{\pm}\, =\, \frac{1}{4}\, \left(\rho_+ \pm \rho_-\right)^2\, .
\end{equation}
This is useful because we frequently use the difference of squares formula in the lemmas of Section \ref{sec:PinskyDiagonal}.

\section{Contour integral alternative derivation of Lemma \ref{lem:gamma2}}
\label{sec:gamma2}

By the equation (\ref{eq:gamma2contour}) we may write
\begin{equation}
	\gamma^{(2)}(x^2,y^2)\, =\, \oint_{\mathcal{C}(0;1)} \oint_{\mathcal{C}(0;1)} \gamma(x\xi,y\zeta)
	\gamma\left(\frac{x}{\xi}\, ,\ \frac{y}{\zeta}\right)\, \frac{d\xi}{2\pi i \xi}\, \cdot \frac{d\zeta}{2\pi i \zeta}\, ,
\end{equation}
which may be rewritten as 
\begin{equation}
	\gamma^{(2)}(x^2,y^2)\, =\, \oint_{\mathcal{C}(0;1)} \oint_{\mathcal{C}(0;1)} \frac{1}{1-(x\xi + y \zeta)}\, 
	\cdot \frac{1}{1-(x\xi^{-1}+y\zeta^{-1})}\, \cdot  \frac{d\xi}{2\pi i \xi}\, \cdot \frac{d\zeta}{2\pi i \zeta}\, .
\end{equation}
This may be rewritten as 
\begin{equation}
\label{eq:gamma2Double}
	\gamma^{(2)}(x^2,y^2)\, =\, -\frac{1}{4\pi^2}\, \oint_{\mathcal{C}(0;1)} \oint_{\mathcal{C}(0;1)} \frac{1}{1-(x\xi + y \zeta)}\, 
	\cdot \frac{1}{\xi \zeta-(x\zeta+y\xi)}\, d\xi\, d\zeta\, .
\end{equation}
Let us first integrate over $\zeta$.
Therefore, we use the partial fraction expansion for $\zeta$:
\begin{equation}
\begin{split}
	\frac{1}{1-(x\xi + y \zeta)}\, 
	\cdot \frac{1}{\xi \zeta-(x\zeta+y\xi)}\,
	&=\, \frac{1}{1-x\xi - y \zeta}\, 
	\cdot \frac{1}{(\xi-x) \zeta-y\xi}\\
	&=\, \frac{1}{(\xi-x)(1-x\xi) - y^2\xi} \left(\frac{y}{1-x\xi - y \zeta} + \frac{\xi-x}{(\xi-x) \zeta-y\xi}\right)
\end{split}
\end{equation}
Therefore, considering the integrand of (\ref{eq:gamma2Double}) as a function of $\zeta$, the poles are at
\begin{equation}
	\alpha\, =\, \frac{1-x\xi}{y}\, ,\ \text{ and }\ \beta\, =\, \frac{y\xi}{\xi-x}\, .
\end{equation}
For $|x|$ and $|y|$ sufficiently small, and for $|\xi|=1$, we see that $|\alpha|>1$ and $|\beta|<1$.
Therefore, from the residue theorem (see for example Theorem 17 in Section 4.5 of Ahlfors \cite{Ahlfors})
\begin{equation}
	\gamma^{(2)}(x^2,y^2)\, =\, \frac{1}{2\pi i}\, \oint_{\mathcal{C}(0;1)} \frac{1}{(\xi-x)(1-x\xi) - y^2\xi}\, d\xi\, .
\end{equation}
Now we may factorize the denominator in the integrand of this integral, as a polynomial in $\xi$ in order to find the locations of the poles
and calculate the residues. Note that if we had $x=0$ the integral $(2\pi i)^{-1} \oint_{\mathcal{C}(0;1)} (\xi - y^2 \xi)^{-1}\, d\xi$ would
give $(1-y^2)^{-1}$ by the residue theorem. Let us assume $x\neq 0$, and then note
\begin{equation}
\begin{split}
	(\xi-x)(1-x\xi) - y^2\xi
	&=\, \xi-x -x \xi^2+x^2\xi- y^2 \xi\\
	&=\, -x \left(\xi^2-\frac{1+x^2-y^2}{x}\, \xi +1\right)
\end{split}
\end{equation}
By the formulas from Section \ref{sec:roots} we have for $k = (1+x^2-y^2)/x$, the roots are at
\begin{equation}
	\xi_{\pm}\, =\, \frac{1}{4}\left(\sqrt{k+2} \pm \sqrt{k-2}\right)^2\,
	=\, \frac{1}{4}\left(\sqrt{\frac{1+x^2-y^2+2x}{x}}\pm\sqrt{\frac{1+x^2-y^2-2x}{x}}\right)^2\, ,
\end{equation}
as long as we assume
\begin{equation}
	k\, =\, \frac{1+x^2-y^2}{x} \in \C \setminus (-2,2)\, .
\end{equation}
Note that we are justified in assuming that $x,y \in [0,\infty)$ and are small, since determining the power series
for $x$ and $y$ in that regime determines it elsewhere by the identity theorem of complex analysis and the 
discussion of the Fubini-Tonelli theorem before, between equations (\ref{eq:betaSecondFormula}) and (\ref{eq:gamma2def})
Note that for $x>0$ sufficiently small, we know that $\xi_+>1$ and $\xi_-<1$. (Note that $\xi_+ \xi_- = 1$.)
Therefore, by the residue theorem, we have
\begin{equation}
	\gamma^{(2)}(x^2,y^2)\, =\, -\frac{1}{x} \cdot \frac{1}{\xi_- - \xi_+}\, =\, \frac{1}{x} \cdot \frac{1}{\xi_+-\xi_-}\, .
\end{equation}
But it is easy to see the algebraic simplification
\begin{equation}
\begin{split}
	x\cdot (\xi_+-\xi_-)\, 
	&=\, \frac{x}{4}\, \left(\left(\sqrt{k+2}+\sqrt{k-2}\right)^2-\left(\sqrt{k+2}-\sqrt{k-2}\right)^2\right)\\
	&=\, \frac{x}{4}\, \left(2\, \sqrt{k+2}\right) \left(2\, \sqrt{k-2}\right)\, .
\end{split}
\end{equation}
Therefore, this may be further simplified as 
\begin{equation}
\begin{split}
	x\cdot (\xi_+-\xi_-)\, 
	&=\, \sqrt{\big((1+x)^2-y^2\big)\big((1-x)^2-y^2\big)}\\
	&=\, \sqrt{\big((1+x^2-y^2)+2x\big)\big((1+x^2-y^2)-2x\big)}\\
	&=\, \sqrt{(1+x^2-y^2)^2-4x^2}
\end{split}
\end{equation}
Then we expand
\begin{equation}
\begin{split}
	(1+x^2-y^2)^2-4x^2\,
	&=\, 1 + 2(x^2-y^2) + (x^2-y^2)^2 - 4x^2\\
	&=\, 1 + (2x^2-2y^2) + (x^4-2x^2 y^2+y^4) - 4x^2\\
	&=\, 1 - 2x^2 -  2 y^2 + x^4 + 2 x^2 y^2 + y^4-4x^2y^2\\
	&=\, 1 - 2 (x^2+y^2) + (x^2+y^2)^2 - 4 x^2 y^2\\
	&=\, (1-(x^2+y^2))^2 - 4 x^2 y^2\, .
\end{split}
\end{equation}
Therefore, we have, 
\begin{equation}
	\gamma^{(2)}(x^2,y^2)\, =\, \frac{1}{\sqrt{\big(1-\left(x^2+y^2\right)\big)^2-4x^2 y^2}}\, .
\end{equation}
Recall that it suffices to assume that we have $x,y \in [0,\infty)$ both sufficiently small to specify the actual power
series in complex $x$ and $y$.
Therefore, we obtain the desired result by replacing $x\geq 0$ by $\sqrt{x}\geq 0$ and replacing $y\geq 0$ by $\sqrt{y}\geq 0$.

\section{Pochhamer identities for $-1/2$ and negative integers}
\label{sec:Poc}

Here we state what needs to be used to derive the equivalence of (\ref{eq:LeftMostBridge}) and (\ref{eq:gamma2combinatorial}).
Note that
\begin{equation}
	(-1/2)_n\, =\, \left(-\frac{1}{2}\right)\left(-\frac{3}{2}\right)\cdots \left(-\frac{1+2(n-1)}{2}\right)\, ,
\end{equation}
which equals
\begin{equation}
	(-1/2)_n\, =\, \frac{(-1)^n}{2^n} \left(1\right)\left(3\right)\cdots \left(2n-1\right)\, ,
\end{equation}
which in turn may be rewritten as 
\begin{equation}
	(-1/2)_n\, =\, \frac{(-1)^n}{2^n} \cdot \frac{(2n)!}{(2)(4)\cdots (2n)}\,
	=\, \frac{(-1)^n}{2^{2n}} \cdot \frac{(2n)!}{n!}\, .
\end{equation}
This well-known fact is easily found, for example on Wikipedia it can be deduced
from the formula $\Gamma(1/2)=\sqrt{\pi}$ and $\Gamma(\frac{1}{2}-n)=(-1)^n 2^{2n} n!\sqrt{\pi}/(2n)!$.

We also know that
for the Pochhammer symbol at a negative integer we have
\begin{equation}
	(-n)_k\, =\, (-n)(-n-1)\cdots(-n-k+1)\,=\, (-1)^k (n)(n+1)(n+k-1)\, 
	=\, (-1)^k\, \frac{(n+k-1)!}{(n-1)!}\, .
\end{equation}
Hence, in particular replacing $n$ by $2n+1$ we have
\begin{equation}
	(-2n-1)_k\, =\, (-1)^k\, \frac{(2n+k)!}{(2n)!}\, .
\end{equation}
Therefore, putting these two together, we see that the summands in (\ref{eq:LeftMostBridge})
are simplified by using the formula
\begin{equation}
	\frac{(-1/2)_n}{n!} \cdot \frac{(-2n-1)_k}{k!}\, =\, 
\frac{(-1)^n}{2^{2n}} \cdot \frac{(2n)!}{\left(n!\right)^2}\, \cdot \frac{(2n+k)!}{(2n)!\, k!}\,
=\,  \frac{(-1)^{n+k}}{2^{2n}} \binom{2n+k}{n,n,k}\, .
\end{equation}
Hence, the left-hand-side of equals (\ref{eq:LeftMostBridge})
\begin{equation}
	\sum_{n=0}^{\infty} \frac{(-1/2)_n}{n!} \cdot \frac{(-2n-1)_k}{k!} \binom{k}{j} (-1)^{n+k} 4^n 
=\, \sum_{n=0}^{\infty} \frac{(-1)^{n+k}}{2^{2n}} \binom{2n+k}{n,n,k}\, \binom{k}{j} (-1)^{n+k} 4^n\, ,\end{equation}
which may be further simplified using
\begin{equation}
	\binom{2n+k}{n,n,k}\, \binom{k}{j}\, =\, \binom{2n+k}{n,n,j,k-j}\, .
\end{equation}
So, the left-hand-side of equals (\ref{eq:LeftMostBridge}) may be rewritten as 
\begin{equation}
	\sum_{n=0}^{\infty} \frac{(-1/2)_n}{n!} \cdot \frac{(-2n-1)_k}{k!} \binom{k}{j} (-1)^{n+k} 4^n 
=\, \sum_{n=0}^{\infty} \binom{2n+k}{n,n,j,k-j}\, .
\end{equation}
In  (\ref{eq:LeftMostBridge})  we wanted to verify that, when we replace
$j\leftarrow \ell-n$ and $k-j\leftarrow m-n$ (using replacement notation as in APL)
then we get $\binom{\ell+m}{\ell,m} \binom{\ell+m}{\ell,m}$.
So this equation may be rewritten as 
\begin{equation}
\sum_{n=0}^{\infty} \binom{2n+\ell+m}{n,n,\ell-n,m-n}\, 
=\, \binom{\ell+m}{\ell,m} \binom{\ell+m}{\ell,m}\, .
\end{equation}
We then take account of the fact that in the multinomial coefficient notation, the multinomial coefficient
equals $0$ unless all the terms in the bottom are nonnegative integers. So we require $\ell-n\geq 0$
and $m-n\geq 0$.
Then the infinite sum for $n \in \{0,1,\dots\}$ may be restricted to $n \in \{0,\dots,\min(\ell,m)\}$.
That is how we reduce equation (\ref{eq:LeftMostBridge}) to (\ref{eq:gamma2combinatorial})
or its even more direct version
\begin{equation}
\sum_{n=0}^{\min(m,n)} \binom{2n+\ell+m}{n,n,\ell-n,m-n}\, 
=\, \binom{\ell+m}{\ell,m} \binom{\ell+m}{\ell,m}\, .
\end{equation}

\section{Elementary proofs from the quadratic formula}
\label{sec:quadratic}

Here we state the proofs of elementary lemmas in Section \ref{sec:PinskyDiagonal} which rely mainly on the quadratic formula
for the roots of a quadratic equation in a single variable.
%For a couple of different reasons, the types of quadratic equations we see here are such that the coefficients are real
%and such that the two roots multiply to $1$. For example, one of the reasons these arise is because these are the properties for the characteristic polynomial of $\mathrm{SL}(2,\R)$ matrices. In turn, those types of matrices arise for us when considering linear fractional transformatons fixing the upper half-plane $\{x+iy \in \C\, :\, x \in \R\, ,\ y>0\}$.
%We need that in Section \ref{sec:Main} because, those types of transformations arise when trying to obtain a standard form for 
%elliptic integrals, as in Section \ref{sec:Main}.
%But the primary reason we meet these types of quadratic equations is that both quartic polynomials $\mathcal{Q}_1(x;\cdot)$
%and $\mathcal{Q}_2(x,w;\cdot)$ split as products of such quadratic equations, which is how we prove the lemmas
%in Section \ref{sec:PinskyDiagonal}.
%
%
%Now 
We prove the elementary lemmas, using the notation introduced in Section \ref{sec:roots}.

\begin{proofof}{\bf Proof of Lemma \ref{lem:calQ1roots}:}
The two factors in (\ref{eq:calQ1Def}) are
\begin{equation}
\label{eq:TwoFactors}
\big(x \cdot(\xi^2+1) - (1+2x)\xi\big)\ \text{ and }\ \big(x \cdot(\xi^2+1) - (1-2x)\xi\big)\, .
\end{equation}
So the roots of $\mathcal{Q}_1(x;\cdot)$ are the solutions of 
\begin{equation}
(\xi^2+1) - \frac{1+2x}{x}\, \cdot\xi\, =\, 0\, ,
\end{equation}
and the solutions of 
\begin{equation}
(\xi^2+1) - \frac{1-2x}{x}\, \cdot\xi\, =\, 0\, .
\end{equation}
Let us denote
\begin{equation}
	\kappa_+(x)\, =\, \frac{1+2x}{x}\ \text{ and }\ \kappa_-(x)\, =\, \frac{1-2x}{x}\, .
\end{equation}
These are both elements of $(2,\infty)$ because we assumed $x \in (0,1/4)$
(so that $1-2x>2x$).
Let us define $\rho^{(+)}_+(x)$, 
$\rho^{(+)}_-(x)$,
$\rho^{(-)}_+(x)$,
and $\rho^{(-)}_+(x)$ as
\begin{equation}
	\rho^{(+)}_{\pm}(x)\, =\, \sqrt{\kappa_+(x) \pm 2}\ \text{ and }\
	\rho^{(-)}_{\pm}(x)\, =\, \sqrt{\kappa_-(x) \pm 2}\, .
\end{equation}
Then we have 
\begin{align}
	c_1(x)\, &=\, \frac{1}{4}\, \left(\rho^{(+)}_{+}(x) - \rho^{(+)}_{-}(x)\right)^2\, ,&
	c_2(x)\, &=\, \frac{1}{4}\, \left(\rho^{(-)}_{+}(x) - \rho^{(-)}_{-}(x)\right)^2\, ,\\
	d_1(x)\, &=\, \frac{1}{4}\, \left(\rho^{(-)}_{+}(x) + \rho^{(-)}_{-}(x)\right)^2\, ,&
	d_2(x)\, &=\, \frac{1}{4}\, \left(\rho^{(+)}_{+}(x) + \rho^{(+)}_{-}(x)\right)^2\, .
\end{align}
(This can be checked by expanding.)
Therefore, by (\ref{eq:xiRoots}), these are the roots of $\mathcal{Q}_1(x;\cdot)$.
Since the leading coefficient of $\mathcal{Q}_1(x;\cdot)$ is $x^2$, the lemma is proved.
\end{proofof}

\begin{proofof}{\bf Proof of Lemma \ref{lem:calQ2roots}:}
The proof is much like the proof of Lemma \ref{lem:calQ1roots}.
The two factors in (\ref{eq:calQ2Def}) are
\begin{equation}
\label{eq:TwoFactors}
\big(x \cdot(\xi^2+1) - (1+\sqrt{4x^2+w^2})\xi\big)\ \text{ and }\ \big(x \cdot(\xi^2+1) - (1-\sqrt{4x^2+w^2})\xi\big)\, .
\end{equation}
So the roots of $\mathcal{Q}_2(x,w;\cdot)$ are the solutions of 
\begin{equation}
(\xi^2+1) - \frac{1+\sqrt{4x^2+w^2}}{x}\, \cdot\xi\, =\, 0\, ,
\end{equation}
and the solutions of 
\begin{equation}
(\xi^2+1) - \frac{1-\sqrt{4x^2+w^2}}{x}\, \cdot\xi\, =\, 0\, .
\end{equation}
Let us denote
\begin{equation}
	\widetilde{\kappa}_+(x,w)\, =\, \frac{1+\sqrt{4x^2+w^2}}{x}\ \text{ and }\ \widetilde{\kappa}_-(x,w)\, =\, \frac{1-\sqrt{4x^2+w^2}}{x}\, .
\end{equation}
These are both elements of $(2,\infty)$ because we assumed $x \in (0,1/4)$ and $0<w<\sqrt{1-4x}$; hence, 
\begin{equation}
	1-\sqrt{4x^2+w^2}\, <\, 2x\, .
\end{equation}
Let us define $\widetilde{\rho}^{(+)}_+(x,w)$, 
$\widetilde{\rho}^{(+)}_-(x,w)$,
$\widetilde{\rho}^{(-)}_+(x,w)$,
and $\widetilde{\rho}^{(-)}_-(x,w)$ as
\begin{equation}
	\widetilde{\rho}^{(+)}_{\pm}(x,w)\, =\, \sqrt{\widetilde{\kappa}_+(x,w) \pm 2}\ \text{ and }\
	\widetilde{\rho}^{(-)}_{\pm}(x,w)\, =\, \sqrt{\widetilde{\kappa}_-(x,w) \pm 2}\, .
\end{equation}
Then we have 
\begin{align}
	\mathcal{A}_1(x,w)\, &=\, \frac{1}{4}\, \left(\widetilde{\rho}^{(+)}_{+}(x,w) - \widetilde{\rho}^{(+)}_{-}(x,w)\right)^2\, ,&
	\mathcal{A}_2(x,w)\, &=\, \frac{1}{4}\, \left(\widetilde{\rho}^{(-)}_{+}(x,w) - \widetilde{\rho}^{(-)}_{-}(x,w)\right)^2\, ,\\
	b_1(x,w)\, &=\, \frac{1}{4}\, \left(\widetilde{\rho}^{(-)}_{+}(x,w) + \widetilde{\rho}^{(-)}_{-}(x,w)\right)^2\, ,&
	b_2(x,w)\, &=\, \frac{1}{4}\, \left(\widetilde{\rho}^{(+)}_{+}(x,w) + \widetilde{\rho}^{(+)}_{-}(x,w)\right)^2\, .
\end{align}
Therefore, by (\ref{eq:xiRoots}), these are the roots of $\mathcal{Q}_2(x,w;\cdot)$.
\end{proofof}
\begin{proofof}{\bf Proof of Lemma \ref{lem:rootsOrder}:}
We use the notation from the proofs of Lemmas \ref{lem:calQ1roots} and Lemma \ref{lem:calQ2roots}.
Considering the function $\Delta : (2,\infty) \to (0,\infty)$ as 
\begin{equation}
	\Delta(\kappa)\, =\, \sqrt{\kappa+2} - \sqrt{\kappa-2}\, ,
\end{equation}
we see that
\begin{equation}
	\Delta'(\kappa)\, =\, - \frac{\sqrt{\kappa+2} - \sqrt{\kappa-2}}{2\sqrt{\kappa^2-4}}\, ,
\end{equation}
which is negative. Therefore, $\Delta$ is a decreasing function. Note that its range is actually $(0,2)$ with $\Delta(\infty) = \lim_{\kappa\to\infty} \Delta(\kappa)=0$. Let us also define $\Delta(0)=2$, which is its value by continuous extension.
Then all this means
\begin{equation}
	\forall \kappa_1, \kappa_2 \in [2,\infty) \cup \{\infty\}\, ,\text{ we have }\ \kappa_1<\kappa_2 \Rightarrow \Delta(\widetilde{\kappa}_2)\, <\, \Delta(\kappa_1)\, .
\end{equation}
Then, by a difference of squares, or monotonicity of the square function on the nonnegative domain
\begin{equation}
	\forall \kappa_1, \kappa_2 \in [2,\infty) \cup \{\infty\}\, ,\text{ we have }\ \kappa_1<\kappa_2 \Rightarrow 
\left(\Delta(\widetilde{\kappa}_2)\right)^2\, <\, \left(\Delta(\kappa_1)\right)^2\, .
\end{equation}
But now notice that
\begin{align}
\notag	0\, &=\, \frac{1}{4}\, \left(\Delta(\infty)\right)^2\, , 
	& \mathcal{A}_1(x,w)\, &=\, \frac{1}{4}\, \left(\Delta(\widetilde{\kappa}_+(x,w))\right)^2\, ,\\[5pt]
	c_1(x)\, &=\, \frac{1}{4}\,  \left(\Delta(\kappa_+(x))\right)^2\, ,
	& c_2(x)\,  &=\, \frac{1}{4}\,  \left(\Delta(\kappa_-(x))\right)^2\, ,\\[5pt]
\notag
	\mathcal{A}_2(x,w)\, &=\, \frac{1}{4}\, \left(\Delta(\widetilde{\kappa}_-(x,w))\right)^2\, ,
	& 1 &=\, \frac{1}{4}\, \left(\Delta(2)\right)^2\, .
\end{align}
Therefore, we will know
\begin{equation}
\label{ineq:rootsOrder}
	0\, <\, \mathcal{A}_1(x,w)\, <\, c_1(x)\, <\, c_2(x)\, <\, \mathcal{A}_2(x,w)\, <\, 1\, ,
\end{equation}
if we can establish
\begin{equation}
\label{ineq:kappaOrder}
	0\, <\, \widetilde{\kappa}_-(x,w)\, <\, \kappa_-(x)\, <\, 
	\kappa_+(x)\, <\, \widetilde{\kappa}_+(x,w)\, <\, \infty\, .
\end{equation}
But clearly $\widetilde{\kappa}_+(x,w)<\infty$ and $\widetilde{\kappa}_-(x,w)>2$ as we showed in the proof of  
Lemma \ref{lem:calQ2roots}.
Now
\begin{equation}
\widetilde{\kappa}_+(x,w)-\kappa_+(x)\, =\, (2x)^{-1}(\sqrt{4x^2+w^2} - 2x)\, >\, 0\, ,
\end{equation}
because we assumed $w>0$.
We have the identical positive quantity for 
\begin{equation}
\kappa_-(x)-\widetilde{\kappa}_-(x,w)\, =\, (2x)^{-1}(\sqrt{4x^2+w^2} - 2x)\, .
\end{equation}
Finally, we note that $\kappa_+(x)-\kappa_-(x)=4$ which is positive. Therefore, equation (\ref{ineq:kappaOrder})
is established. Now the order 
\begin{equation}
	1\, <\, b_1(x,w)\, <\, d_1(x)\, <\, d_2(x)\, <\, b_2(x,w)\, <\, \infty\, ,
\end{equation}
follows from (\ref{ineq:rootsOrder}) and equations (\ref{eq:Q1inversive}) and (\ref{eq:Q1inversive}).
\end{proofof}
\begin{proofof}{\bf Proof of Lemma \ref{lem:sqrtContinuation}:}
To see that $f(x;\cdot)$ is analytic, let us define $h(x;\cdot) : \C\setminus\{0\} \to \C$ as
\begin{equation}
	h(x;\xi)\, =\, \left(1 - x \cdot(\xi + \xi^{-1})\right)^2-4x^2\, .
\end{equation} 
Then we may write $f(x;\xi) = \sqrt{h(x;\xi)}$, which is analytic, as long as we prove
\begin{equation}
h(x;\cdot)\big(\mathcal{U}(0;d_1(x)) \setminus \overline{\mathcal{U}}(0;c_1(x))\big)\, =\, 
\left\{h(x;\xi)\, :\, \mathcal{U}(0;d_1(x)) \setminus \overline{\mathcal{U}}(0;c_1(x))\right\}\,
\end{equation} 
is a subset of 	$\C \setminus (-\infty,0]$.
Since $4x^2+1 \in \R$, then $h(x;\xi) \in (-\infty,0]$ only if
\begin{equation}
1 - x \cdot(\xi+\xi^{-1}) \in \R \cup i\R\, ,
\end{equation}
where $i\R$ means $\{iy\, :\, y \in \R\}$. If we decompose $\xi$ into its modulus and phase as 
\begin{equation}
	\xi\, =\, r e^{i\theta}\, ,
\end{equation}
for $r \in [0,\infty)$ and $\theta \in \R$, then we have
\begin{equation}
\label{eq:hArg}
1 - x \cdot(\xi+\xi^{-1})\, =\, 1 - x\cdot (r+r^{-1}) \cos(\theta) - i x \cdot (r-r^{-1}) \sin(\theta)\, .
\end{equation}
So, in order to have this quantity be in $\R$ requires $r=1$ or $\theta \in \pi \Z$, which means $|\xi|=1$
or $\xi \in \R$.
To have the right-hand-side of (\ref{eq:hArg}) be an element of $i\R$ would require $x \cdot (r+r^{-1})\cos(\theta)=1$.

Let us first consider the case that (\ref{eq:hArg}) is an element of $i\R$. This requires
\begin{equation}
\label{ineq:ImaginaryCondition}
	r+r^{-1}\, \geq\, \frac{1}{x}\, ,
\end{equation}
and $\cos(\theta) = 1/\big(x\cdot (r+r^{-1})\big)$.
The inequality in (\ref{ineq:ImaginaryCondition}) may be rewritten as $r^2+1-\frac{1}{x}\, \cdot r\geq 0$. 
Consulting  Appendix  \ref{sec:roots}, we see that
\begin{equation}
	r^2+1-\frac{1}{x}\, \cdot r\, =\, \left(r-\mathfrak{R}_+(x)\right)\left(r-\mathfrak{R}_-(x)\right)\, ,
\end{equation}
for the two real roots:
\begin{equation}
\mathfrak{R}_+(x)\, =\, \frac{1+\sqrt{1-4x^2}}{2x}\quad \text{ and }\quad
\mathfrak{R}_-(x)\, =\, \frac{1-\sqrt{1-4x^2}}{2x}\, .
\end{equation}
They are real and $0<\mathfrak{R}_-(x)<\mathfrak{R}_+(x)$ 
because we have assumed that $x \in (0,1/4)$. 
(It would have sufficed to have $0<x<1/2$.)
Therefore, equation (\ref{ineq:ImaginaryCondition}) is equivalent to 
\begin{equation}
	r \in \R \setminus \left(\mathfrak{R}_-(x),\mathfrak{R}_+(x)\right)\, .
\end{equation}
That is the condition for the left-hand-side of equation (\ref{eq:hArg})
to be an element of $i\R$,
as long as $\xi= r e^{i\theta}$ where $\cos(\theta) = 1/\big(x\cdot(r+r^{-1})\big)$.
In other words,
\begin{equation}
\label{eq:conditionImag}
\forall \xi \in \C\, ,\ \text{ we have }\ |\xi| \in (\mathfrak{R}_-(x),\mathfrak{R}_+(x)) \Rightarrow 1-x\cdot(\xi+\xi^{-1}) \not\in i\R
\end{equation}

Next, let us change attention to the condition for $1-x\cdot (\xi+\xi^{-1})$ to be in $\R$.
We must have $|\xi|=1$ or $\xi \in \R$.
But if $|\xi|=1$ then $1-x\cdot (\xi+\xi^{-1}) = 1-2x\cdot \cos(\theta)>2x$ since we assumed $x \in (0,1/4)$.
So in this case $h(x;\xi)$ is in $\mathcal{D}_1$ from Remark \ref{rem:sqrt}.
So, instead, consider the case that $\xi$ is in $\R$. Then we prefer changing  the notation
to $\xi=t \in \R$ (which could potentially be negative, which is why we avoid recycling the letter $r$ which was
positive before). Then to have $h(x;t)\leq 0$ requires
\begin{equation}
	h(x,t)\, =\, \frac{\mathcal{Q}_1(x;t)}{t^2}\, \leq\, 0\, ,
\end{equation}
which requires, by Lemma \ref{lem:calQ1roots},
\begin{equation}
	t \in \R \setminus \left(c_2(x),d_1(x)\right)\, .
\end{equation}
So combining this with (\ref{eq:conditionImag}) we have
\begin{equation}
\label{eq:condRealImag}
	\forall \xi \in \C\, ,\ \text{ we have }\ |\xi| \in  \left(c_2(x),d_1(x)\right) \cap 
\left(\mathfrak{R}_-(x),\mathfrak{R}_+(x)\right) \Rightarrow h(x;\xi) \in \mathcal{D}_1\, .
\end{equation}
It is easy to see that 
\begin{equation}
	\mathfrak{R}_-(x)<c_2(x)<d_1(x)<\mathfrak{R}_+(x)\, ,
\end{equation}
for $x \in (0,1/4)$. For instance $1-2x\geq \sqrt{1-4x}$
hence $\sqrt{1-4x^2}\geq 2x+\sqrt{1-4x}$ because
\begin{equation}
(1-2x)(1+2x)\, \geq\, (1-2x)^2 + 4x \sqrt{1-4x}\, =\, (2x+\sqrt{1-4x})^2\, .
\end{equation}
Therefore, for analyticity of $\sqrt{h(x;\cdot)}$ equation (\ref{eq:condRealImag}) reduces to just $|\xi| \in (c_2(x),d_1(x))$.
\end{proofof}

\section{Elementary proofs from contour deformations}
\label{sec:Deform}

Here we will prove Lemmas \ref{lem:sqrtContinuationB} and \ref{lem:sqrtContinuationC}.

\begin{proofof}{\bf Proof of Lemma \ref{lem:sqrtContinuationB}:}
Note that if $\xi \in (c_2(x),d_1(x))$ then by Lemma \ref{lem:rootsOrder} we have
\begin{equation}
\label{eq:xiOrder}
c_1(x)<c_2(x)<\xi<d_1(x)<d_2(x)\, .
\end{equation}
Therefore, for $\mathcal{D}_1$ as in Remark \ref{rem:sqrt},
\begin{equation}
\label{eq:calD1condition}
	\forall \xi \in (c_2(x),d_1(x))\, ,\ \text{ we have }\ \xi - c_1(x)\, ,\ \xi-c_2(x)\, ,\ d_1(x)-\xi\, ,\ d_2(x)-\xi\, \in\, \mathcal{D}_1\, .
\end{equation}
Similarly, if $\xi \in \C \setminus \R$ then we also have the conditions on the right-hand-side of  (\ref{eq:calD1condition})
because the differences above are also in $\C \setminus \R$
(because $c_1(x)$, $c_2(x)$, $d_1(x)$ and $d_2(x)$ are all in $\R$).
So this proves that for $\xi \in \mathcal{D}_3(x)$, we have the conditions on the right-hand-side
of (\ref{eq:calD1condition}).
So $g(x;\cdot)$ is analytic, as given by the formula in (\ref{eq:gDefinition}).

Note that $\big(g(x;\xi)\big)^2 = \mathcal{Q}_1(x;\xi)$ by Lemma \ref{lem:calQ1roots},
equation (\ref{eq:calQ1roots}).
Also note that $1$ is in $(c_2(x),d_1(x))$ by Lemma \ref{lem:rootsOrder} .
So $g(x;1) \in (0,\infty)$ by (\ref{eq:gDefinition}) and (\ref{eq:xiOrder}).
Since $\big(g(x;\xi)\big)^2$ and $\big(f(x;\xi)\big)^2$ satisfy
\begin{equation}
	\big(g(x;\xi)\big)^2\, =\, \mathcal{Q}_1(x;\xi)\ \text{ and }\ \big(f(x;\xi)\big)^2\, =\, \frac{\mathcal{Q}_1(x;\xi)}{\xi^2}\, ,
\end{equation}
and since $g(x;1)$ and $f(x;1)$ have the same phase (they are both real and positive)
we see that $g(x;\xi)/\xi$ and $f(x;\xi)$ are equal, for $\xi$ in a neighborhood of $1$. 
But then
define the domain $\widetilde{\mathcal{D}}_3(x)$ by
\begin{equation}
	\widetilde{\mathcal{D}}_3(x)\, =\, \mathcal{D}_2(x) \cap \mathcal{D}_3(x)\, 
	=\,  \big(\mathcal{U}(0;d_1(x)) \setminus \overline{\mathcal{U}}(0;c_1(x))\big) \setminus (c_1(x),d_1(x))\, .
\end{equation}
By analytic continuation this means
\begin{equation}
\label{eq:Equalityfg}
	\forall \xi \in \widetilde{\mathcal{D}}_3(x)\, ,\
	\text{ we have }\ g(x;\xi)\, =\, \xi \cdot f(x;\xi)\, .
\end{equation}
This follows because $\widetilde{\mathcal{D}}_3(x)$ is a connected,
simply connected set. Therefore, there is unique analytic continuation of both $\xi \mapsto g(x;\xi)$ and $\xi \mapsto \xi \cdot f(x;\xi)$
from a neighborhood of $1$ to the entire domain $\widetilde{\mathcal{D}}_3(x)$.
So $f(x;\xi)=g(x;\xi)/\xi$ on that whole domain.
\end{proofof}
\begin{proofof}{\bf Proof of Lemma \ref{lem:sqrtContinuationC}:}
We note that by continuity of the analytic function $g$, and the fact that $\mathcal{D}_4(x)$
is a subset of the closure of $\mathcal{D}_3(x)$ (which is all of $\C$),
we may give the alternative formula
\begin{equation}
\label{eq:tildegDef4}
	\forall \xi \in \mathcal{D}_4(x)\, ,\ \text{ we have }\ \widetilde{g}(x;\xi)\, =\, 
\lim_{\substack{\zeta \to \xi\\ \zeta \in \mathcal{D}_3(x)}}\, g(\zeta)\, .
\end{equation}
What we mean by saying this is the following: equations (\ref{eq:tildegDef2}) and 
(\ref{eq:tildegDef3}) are equivalent to  (\ref{eq:tildegDef4}).
In order to see that  for $\widetilde{g}$  the limits are well-defined on all of $\mathcal{D}_4(x)$,  note that for $\xi \in (-\infty,c_1(x))$ we have for $\zeta=\xi+i \epsilon$ with $\epsilon>0$,
that 
\begin{equation}
	\sqrt{\zeta-c_1(x)}\, =\, \sqrt{-(c_1(x)-\xi)+i\epsilon}\, =\, i\, \sqrt{c_1(x)-\xi}\, + \frac{\epsilon}{\sqrt{c_1(x)-\xi}} + O(\epsilon)^2\, ,
\end{equation}
and
\begin{equation}
	\sqrt{\zeta-c_2(x)}\, =\, \sqrt{-(c_2(x)-\xi)+i\epsilon}\, =\, i\, \sqrt{c_2(x)-\xi}\, + \frac{\epsilon}{\sqrt{c_2(x)-\xi}} + O(\epsilon)^2\, ,
\end{equation}
as $\epsilon \to 0^+$.
Therefore, we have 
\begin{equation}
\lim_{\epsilon \to 0^+} g(x;\xi+i\epsilon)\, =\, x \cdot \left(i\, \sqrt{c_1(x)-\xi}\right)
\left(i\, \sqrt{c_2(x)-\xi}\right)\, \cdot \sqrt{d_1(x)-\xi}\, \cdot \sqrt{d_2(x)-\xi}\, .
\end{equation}
But, similarly, we will have
\begin{equation}
\lim_{\epsilon \to 0^+} g(x;\xi-i\epsilon)\, =\,  x \cdot \left(-i\, \sqrt{c_1(x)-\xi}\right)
\left(-i\, \sqrt{c_2(x)-\xi}\right)\, \cdot \sqrt{d_1(x)-\xi}\, \cdot \sqrt{d_2(x)-\xi}\, .
\end{equation}
Since $i^2 = (-i)^2 = -1$ these two limits are equal. (The same will not be true for points $\xi \in (c_1(x),c_2(x))$, as we will see shortly.) 
Therefore, the limit exists and equals
\begin{equation}
	\widetilde{g}(x;\xi)\, =\, - x \cdot \sqrt{c_1(x)-\xi}\, 
\cdot \sqrt{c_2(x)-\xi}\, \cdot \sqrt{d_1(x)-\xi}\, \cdot \sqrt{d_2(x)-\xi}\, .
\end{equation}
A similar argument shows that if instead we had $\xi \in (d_2(x),\infty)$ then the limit $\lim_{\zeta \to \xi} g(x;\zeta)$ would exist and would equal
\begin{equation}
	\widetilde{g}(x;\xi)\, =\, - x \cdot \sqrt{\xi-c_1(x)}\, 
\cdot \sqrt{\xi-c_2(x)}\, \cdot \sqrt{\xi-d_1(x)}\, \cdot \sqrt{\xi-d_2(x)}\, .
\end{equation}
By elementary results of complex analysis, as in \cite{Ahlfors}, it is easy to see that $\widetilde{g}(x;\cdot)$
is analytic  on all of $\mathcal{D}_4(x)$ since it is the limit of the analytic function
$g(x;\cdot)$ which is continuous across $(-\infty,c_1(x))$ and $(d_2(x),\infty)$.
For example, using the Cauchy integral formula and Morera's theorem and the fact that the contour
integrals are continuous across $(-\infty,c_1(x))$ and $(d_2(x),\infty)$, one can check the integral conditions to be analytic
(and there are undoubtedly more elegant approaches using Riemann surfaces).
Now $\widetilde{g}(x;\cdot)/\xi$ equals $f(x;\cdot)$ on 
$\widetilde{\mathcal{D}}_3(x)$
by equations (\ref{eq:gEQUALSf}) and (\ref{eq:tildegDef2}), or for example (\ref{eq:Equalityfg}).
But by continuity, the equality can be extended to all of $\mathcal{D}_2(x)$ since 
$\mathcal{D}_2(x) \subseteq \mathcal{D}_4(x)$.
\end{proofof}
%
%\begin{remark}
%We frequently refer to the argument of the square-root: if we have a term of the form $\sqrt{f(x)}$ then
%the ``argument of the square-root'' means $f(x)$. That is common language as in $y=f(x)$, the ``argument of $f$'' is the independent variable $x$. But in complex analysis, the word ``argument'' may also refer to the imaginary part of the logarithm function. (See for example, Section 3.4 of \cite{Ahlfors}.) For that concept, we will instead use the word ``phase.'' The ``phase'' means $e^{i\varphi}$ when we decompose
%a complex number as $z=R e^{i \varphi}$ for $R\geq 0$ and $\varphi \in \R$.
%\end{remark}

\section{Proof of lemmas related to the residue theorem}
\label{sec:residue}

\begin{proofof}{\bf Proof of Lemma \ref{lem:residue1}:}
For any $\epsilon\in (0,b_1(x,w)-1)$, the integrand,
\begin{equation}
\label{eq:IntegrandFirst}
\mathcal{I}(\xi)\, =\, \frac{w\cdot \xi+\widetilde{g}(x;\xi)}
{x^2\cdot \big(\xi  - \mathcal{A}_1(x,w)\big)\big(\xi-\mathcal{A}_2(x,w)\big)\big(\xi-b_1(x,w)\big)\big(\xi-b_2(x,w)\big)}\, 
\cdot \frac{1}{2\pi i}\, ,
\end{equation}
is analytic in the open domain $\mathcal{U}(0;1+\delta) \setminus \Big(\{\mathcal{A}_1(x,w),\mathcal{A}_2(x,w)\}\cup [c_1(x),c_2(x)]\Big)$. The value of the contour integral is independent of the choice of rectifiable arc/contour, or in our case, independent of the sum of contours as long as that sum is  homotopic
to $\mathcal{C}(0;1)$ in the domain of analyticity. (See for example \cite{Ahlfors} or \cite{PemantleWilson}.) Therefore,
for $\delta \in (0,1)$ sufficiently small we deform to the union of three contours
\begin{gather}
\Gamma_1(\delta)\, =\, \mathcal{C}(\mathcal{A}_1(x,w);\delta)\ \text{ and }\ \Gamma_2(\delta)\, =\, \mathcal{C}(\mathcal{A}_2(x,w);\delta)\, ,\\[5pt]
\notag
\Gamma_3(\delta) = \Gamma_1(\delta) + \Gamma_2(\delta) + \gamma_3(\delta) + \gamma_4(\delta)\, ,\\
\notag
\gamma_k(\delta)\, :\, \xi(t)\, =\, \widetilde{c}_k(\delta) \cdot (1-t) + \widetilde{c}_{k+1}(\delta) \cdot t\, ,\\
\text{ for }\ 0\leq t\leq 1\, ,\ \text{ for }\ k \in \{1,2,3,4\}\, ,\\
\notag 
\widetilde{c}_1(\delta)=\widetilde{c}_5(\delta)=c_1(x,w)-\delta\, ,\qquad \widetilde{c}_2(\delta)=c_2(x,w)-\delta\, ,\\
\notag
 \widetilde{c}_3(\delta)=c_2(x,w)+\delta\, ,\qquad \widetilde{c}_4(\delta)=c_1(x,w)+\delta\, .
\end{gather}
See Figure \ref{fig:sqrtNew}, cited previously. (Note that one can add and subtract line segments along the real axis external to these three contours
and internal to $\mathcal{C}(0;1)$ joining these contours to each other, and to $\mathcal{C}(0;1)$. In the appropriate sense,
the sum of these three contours is homotopic to $\mathcal{C}(0;1)$.)
Technically the rectangle $\Gamma_3(\delta)$ does intersect the boundary of the domain at 
$(\widetilde{c}_2(\delta)+\widetilde{c}_3(\delta))/2=c_2(x)$ and 
at $(\widetilde{c}_4(\delta)+\widetilde{c}_5(\delta))/2=c_1(x)$.
But $\mathcal{I}$
is continuous at those points so that this introduces no difficulty.
(It suffices to consider $\Gamma_3(\delta)$ as a limit of contours strictly inside the domain of analyticity.)

For $\Gamma_1(\delta)$ and $\Gamma_2(\delta)$ as we shrink $\delta$ we obtain the residues (as in Theorem 6, page 119 of \cite{Ahlfors})
\begin{equation}
\label{eq:ResiduesFirst}
	\forall k \in \{1,2\}\ \text{ we have }\
	\lim_{\delta \to 0^+} \oint_{\Gamma_k(\delta)} \mathcal{I}(\xi)\, d\xi\, 
	=\, 2 \pi i\, \cdot \lim_{\xi \to a_k(x,w)} \big(\xi - a_k(x,w)\big)\cdot \mathcal{I}(\xi)\, ,
\end{equation}
where we use the slightly unusual formula for the residues because,
in the lemmas of the last section, we have already established that there are simple poles
at $\mathcal{A}_1(x,w)$ and $\mathcal{A}_2(x,w)$. I.e., there are no higher-order poles.
Combining the two contributions from equation (\ref{eq:ResiduesFirst}) gives the formula for $\mathcal{A}_1(x,w)$ in (\ref{eq:A1formula}).

For the contribution to the integral coming from the rectangle $\Gamma_3(\delta)$, we may neglect the left and right sides $\Gamma_2(\delta)$
and $\gamma_3(\delta)$ because their arclengths vanish in the limit and $\mathcal{I}$ is continuous at $c_1(x)$ and $c_2(x)$.
The effect of taking the $\delta \to 0^+$ limit for $\Gamma_1(\delta)$ and $\Gamma_2(\delta)$ is as follows:
\begin{equation}
\label{eq:BranchFirst}
	\lim_{\delta \to 0^+} \oint_{\Gamma_3(\delta)} \mathcal{I}(\xi)\, d\xi\, 
	=\, \int_{c_1(x)}^{c_2(x)} \big(\mathcal{I}(\xi-i\cdot 0) - \mathcal{I}(\xi+i\cdot 0)\big)\, d\xi\, ,
\end{equation}
where, as is standard convention, we denote
\begin{equation}
\mathcal{I}(\xi+i\cdot 0)\, =\, \lim_{\delta\to 0^+} \mathcal{I}(\xi+i\cdot \delta)\ \text{ and }
\mathcal{I}(\xi-i\cdot 0)\, =\, \lim_{\delta\to 0^+} \mathcal{I}(\xi-i\cdot \delta)\, .
\end{equation}
In the next paragraph, and following, we will show that
the limits from the upper half plane and from the lower half plane do exist, that they are continuous functions on the interval
$[c_1(x),c_2(x)]$ and that the limits are obtained uniformly on $[c_1(x),c_2(x)]$ as $\delta \to 0^+$. Therefore, there
are no technical
difficulties in interchanging the order of taking the $\delta \to 0^+$ limits and taking the integral, using the dominated convergence theorem.
The reason that the limiting integrand is $\mathcal{I}(\xi-i\cdot 0) - \mathcal{I}(\xi+i\cdot 0)$ (instead of its opposite)
is because of the positive orientation of the contour $\Gamma_3(\delta)$.

Before calculating the limits $\mathcal{I}(\xi \pm i \cdot 0)$, we note that the integrand may be written as 
\begin{equation}
\label{eq:IntegrandSecond}
	\mathcal{I}(\xi)\, =\, \frac{w \cdot \xi + \widetilde{g}(x;\xi)}{2 \pi i  \mathcal{Q}_2(x,w;\xi)}\,,
\end{equation}
as an alternative to equation (\ref{eq:IntegrandFirst}).
But then the term $w\cdot \xi  / \mathcal{Q}_2(x,w;\xi)$ is analytic in a neighborhood of $[c_1(x),c_2(x)]$
since it merely has poles at $\mathcal{A}_1(x,w)$ and $\mathcal{A}_2(x,w)$ within $\mathcal{U}(0;1)$, and we have established
that $\mathcal{A}_1(x,w)<c_1(x)$ and $c_2(x)<\mathcal{A}_2(x,w)$ in the lemmas of the last section.
Therefore,
\begin{equation}
\label{eq:IntegrandRedefinition}
\mathcal{I}(\xi\pm i \cdot 0)\, =\, \frac{w \cdot \xi}{2 \pi i  \mathcal{Q}_2(x,w;\xi)} + \widetilde{\mathcal{I}}(\xi \pm i \cdot 0)\, ,
\end{equation}
where we define
\begin{equation}
\widetilde{\mathcal{I}}(\xi)\, =\, \frac{\widetilde{g}(x;\xi)}{2 \pi i  \mathcal{Q}_2(x,w;\xi)}\, .
\end{equation}
Clearly the limits of the first summand on the right hand side of (\ref{eq:IntegrandRedefinition}) are the same, are continuous, and are attained
uniformly by Weierstrass's M-test for local uniform convergence (see Section 2.3 of \cite{Ahlfors}) since that function is analytic in an open neighborhood of the whole interval $[c_1(x),c_2(x)]$.
\end{proofof}

\begin{proofof}{\bf Proof of :}
Now considering $\widetilde{\mathcal{I}}$, the denominator is also analytic an non-vanishing in a neighborhood of the interval $[c_1(x),c_2(x)]$, for the same reason as
before: that $\mathcal{A}_1(x,w)<c_1(x)$ and $c_2(x)<\mathcal{A}_2(x,w)$. Therefore, 
\begin{equation}
	\widetilde{\mathcal{I}}(\xi \pm i \cdot 0)\, =\, \frac{1}{2 \pi i \mathcal{Q}_2(x,w;\xi)}\,  \cdot \widetilde{g}(x;\xi\pm i 0)\, .
\end{equation}
Finally, we turn attention to the key point of the calculation, which is the jump discontinuity of the square-root type functon $\widetilde{g}(x;\xi \pm i 0)$
across its branch-cut $[c_1(x),c_2(x)]$.
Since we are going to focus on $\xi \in [c_1(x),c_2(x)]$ we change the symbol to $r$ to indicate better that it is a real number.
That frees the symbol $\xi$ to be the point $\xi = r + i \delta$ (or $\xi = r-i\delta$) for the complex variable converging to $r$
from the upper half-plane (or lower half-plane).

Suppose that we have
$r \in (c_1,c_2)$ and $\delta >0$.
We are seeking the limits from both the upper half-plane and lower half-plane of the phases, which we denote as
\begin{equation}
	\zeta_{\pm}(r)\, =\, \lim_{\delta \to 0^+} \frac{\widetilde{g}(x;r \pm i \delta)}{|\widetilde{g}(x;r \pm i \delta)|}\, .
\end{equation}
We focus on just the limit from the upper half-plane, first, since the other limit is obtained by similar (symmetric) considerations.
We  may refer to Figure \ref{fig:sqrtNew}
to determine the phase of $\widetilde{g}(x;\xi)$ for $\xi = r + i \epsilon$.
We have
\begin{equation}
	\xi - c_1(x) \, =\, \sqrt{\big(r-c_1(x)\big)^2+\delta^2}\, \cdot e^{i\theta_1}\ \text{ and }\
	\xi - c_2(x) \, =\, \sqrt{\big(c_2(x)-r\big)^2+\delta^2}\, \cdot  e^{i\theta_2}\, ,	
\end{equation}
where $\theta_1 \in (0,\pi/2)$ and $\theta_1 \in (\pi/2,\pi)$.
But in the limit $\delta\to 0^+$, we have $\theta_1 \to 0$
while $\theta_2\to\pi$. So we have
\begin{equation}
	\lim_{\delta \to 0^+} \frac{\xi - c_1(x)}{\left|\xi - c_1(x)\right|}\, =\, 1\ \text{ and }
	\lim_{\delta \to 0^+} \frac{\xi - c_2(x)}{\left|\xi - c_2(x)\right|}\, =\, -1\, .
\end{equation}
But then taking the square-root, we have $e^{i\theta_1/2}$ converges to $1$
while $e^{i\theta_2/2}$ converges to $i$.
Multiplying these together, we get 
\begin{equation}
\forall r \in (c_1,c_2)\ \text{ we have }\ \zeta_+(r)\, =\, i\, .
\end{equation}
Note that $d_1-r>0$ and $d_2-r>0$ for $r$ throughout $[-1,1]$, which is why we did not consider those
factors.
Of course if we replace $\epsilon$ by $-\epsilon$ then the angles are reflected:
$\theta_1 \leftarrow -\theta_1$ and $\theta_2 \leftarrow -\theta_2$ using replacement notation
(as in APL). So the phase changes from $\zeta_+(r)=i$ to its reciprocal
\begin{equation}
\forall r \in (c_1,c_2)\ \text{ we have }\ \zeta_-(r)\, =\, \frac{1}{i}\, =\, -i\, .
\end{equation}
From this we see that
\begin{equation}
	\widetilde{g}(x;r\pm i\cdot 0)\, =\, \zeta_{\pm}(r)\, \lim_{\delta \to 0} |\widetilde{g}(x;r+i\cdot \delta)|\, ,
\end{equation}
where the final limit of the absolute value occurs uniformly and is the same from both the upper and lower half-planes.
More precisely,
\begin{equation}
	\widetilde{g}(x;r\pm i\cdot 0)\, =\, x \zeta_{\pm}(r)\, 
\cdot \sqrt{r-c_1(x)}\, \cdot \sqrt{c_2(x)-r}\,
\cdot \sqrt{d_1(x)-r}\, \cdot \sqrt{d_2(x)-r}\, .
\end{equation}
Since $\zeta_{\pm}(r) = \pm i$, combining this with (\ref{eq:BranchFirst}) gives the formula (\ref{eq:A2formula}) for the contribution to the contour integral
arising from $\Gamma_3(\delta)$.
Note that $\zeta_-(r) - \zeta_+(r) = -2i$. This cancels the $2i$ in the $2\pi i$ with the quotient being equal to $-1$.
But since $\mathcal{Q}_2(r)$ is negative on the interval $[c_1(x),c_2(x)]$ (owing to the fact, stated several times before, that $\mathcal{A}_1(x,w)<c_1(x)$
and $c_2(x)<\mathcal{A}_2(x,w)$) we have put the $-1$ with that term in (\ref{eq:A2formula})  so that the integrand on $[c_1(x),c_2(x)]$ is manifestly positive.
\end{proofof}

\section{Elliptic integrals and linear fractional transformations}
\label{sec:Elliptic}

From equation (\ref{eq:A2formula2}) we know that we may write $\mathcal{A}_2(x,w)$ as
\begin{equation}
\label{eq:A2formula2b}
	\mathcal{A}_2(x,w)\, =\, \frac{1}{\pi}\, \int_{c_1(x)}^{c_2(x)} \frac{1}{\sqrt{-\mathcal{Q}_1(x;r)}} \cdot
\frac{\mathcal{Q}_1(x;r)}{\mathcal{Q}_2(x,w;r)}\, dr\, ,
\end{equation}
where we have written the square-root of the quartic in the denominator, which is what is typically
done for elliptic integrals.
The next step for elliptic integrals is to use linear fractional transformations to bring the
elliptic integral in the Legendre normal form. 
In our case, this means the following.
Suppose we find a non-singular linear fractional transformation $\Phi : \C \cup \{\infty\} \to \C \cup \{\infty\}$ defined as
\begin{equation}
\label{eq:PhiDef}
	\Phi(z)\, =\, \frac{\mathfrak{A} z+\mathfrak{B}}{\mathfrak{C}z+\mathfrak{D}}\, ,
\end{equation}
where $\mathfrak{A}\mathfrak{D}-\mathfrak{B}\mathfrak{C}\neq 0$, which is then invertible, 
with inverse  $\Phi^{-1} : \C \cup \{\infty\} \to \C \cup \{\infty\}$
given by
\begin{equation}
\label{eq:PhiInvDef}
	\Phi^{-1}(\xi)\, =\, \frac{\mathfrak{D}\xi-\mathfrak{B}}{-\mathfrak{C}\xi+\mathfrak{A}}\, .
\end{equation}
Then if we take $r=\Phi^{-1}(s)$, and if we know
\begin{equation}
\mathfrak{A}\mathfrak{D}-\mathfrak{B}\mathfrak{C}\, >\, 0\ \text{ and }\ 
\Phi(c_1(x))<\Phi(c_2(x))\, ,
\end{equation}
then we may change variables, using $r=\Phi^{-1}(s)$ to get
\begin{equation}
\label{eq:A2formula2c}
	\mathcal{A}_2(x,w)\, =\, \frac{1}{\pi}\, \int_{\Phi(c_1(x))}^{\Phi(c_2(x))} \frac{1}{\sqrt{-\mathcal{Q}_1(x;\Phi^{-1}(s))}} \cdot
\frac{\mathcal{Q}_1(x;\Phi^{-1}(s))}{\mathcal{Q}_2(x,w;\Phi^{-1}(s))}\, 
\frac{\mathfrak{A}\mathfrak{D}-\mathfrak{B}\mathfrak{C}}{(-\mathfrak{C}s+\mathfrak{A})^2}\, ds\, ,
\end{equation}
where the last factor comes from taking the derivative $dr = d\Phi^{-1}(s)$.
We rewrite this as follows
\begin{equation}
\label{eq:A2formula2c}
	\mathcal{A}_2(x,w)\, =\, \frac{1}{\pi}\, \int_{\Phi(c_1(x))}^{\Phi(c_2(x))} 
\frac{\mathcal{Q}_1(x;\Phi^{-1}(s))}{\mathcal{Q}_2(x,w;\Phi^{-1}(s))}
\cdot \frac{1}{\sqrt{-\mathcal{Q}_1(x;\frac{\mathfrak{D}s-\mathfrak{B}}{-\mathfrak{C}s+\mathfrak{A}}) \cdot (-\mathfrak{C}s+\mathfrak{A})^4}}\, ds\, ,
\end{equation}
because then we choose $a,b,c,d \in \R$ such that the term inside the square-root becomes
\begin{equation}
\label{eq:Q1Transformation}
-\mathcal{Q}_1\left(x;\frac{\mathfrak{D}s-\mathfrak{B}}{-\mathfrak{C}s+\mathfrak{A}}\right) 
\cdot (-\mathfrak{C}s+\mathfrak{A})^4\, =\, 
-x^2 \prod_{\rho \in \{c_1(x),c_2(x),d_1(x),d_2(x)\}} 
\Big(\mathfrak{D}s-\mathfrak{B} - \big(-\mathfrak{C}s+\mathfrak{A}\big)\rho\Big)\, ,
\end{equation}
where we used (\ref{eq:calQ1roots}).
But we know the quartic function in $s$ on the right hand side of 
(\ref{eq:Q1Transformation}) is such that it vanishes at $\Phi(c_1(x))$, $\Phi(c_2(x))$, $\Phi(d_1(x))$
and $\Phi(d_2(x))$. Therefore, by the fundamental theorem of algebra, we have
\begin{equation}
\label{eq:Q1Transformation2}
-\mathcal{Q}_1\left(x;\frac{\mathfrak{D}s-\mathfrak{B}}{-\mathfrak{C}s+\mathfrak{A}}\right) 
\cdot (-\mathfrak{C}s+\mathfrak{A})^4\, =\, 
-x^2 \prod_{\rho \in \{c_1(x),c_2(x),d_1(x),d_2(x)\}} 
\Big(\big(\mathfrak{D}+\mathfrak{C}\rho\big)\cdot
\big(s-\Phi(\rho)\big)\, .
\end{equation} 
If we choose $\Phi$ such that
\begin{equation}
\label{eq:PhiCondition}
	\Phi(c_1(x))\, =\, -1\, ,\ \Phi(c_2(x))\, =\, 1\, ,\ \Phi(d_1(x))\, =\, 1/k\, ,\ 
	\Phi(d_2(x))\, =\, -1/k\, ,
\end{equation}
for some number $k \in (0,1)$, then we will have the desired form for the Legendre convention,
\begin{equation}
\label{eq:Q1Transformation2}
-\mathcal{Q}_1\left(x;\frac{\mathfrak{D}s-\mathfrak{B}}{-\mathfrak{C}s+\mathfrak{A}}\right) 
\cdot (-\mathfrak{C}s+\mathfrak{A})^4\, =\, 
\Xi(x) \cdot (1-s^2)(1-k^2s^2)\, ,
\end{equation} 
where the constant $\Xi(x)$ may be calculated as
\begin{equation}
	\Xi(x)\, =\, -\frac{x^2}{k^2}\, \cdot 
	\big(\mathfrak{D}+\mathfrak{C}\cdot c_1(x)\big) \cdot
	\big(\mathfrak{D}+\mathfrak{C}\cdot c_2(x)\big) \cdot
	\big(\mathfrak{D}+\mathfrak{C}\cdot d_1(x)\big) \cdot
	\big(\mathfrak{D}+\mathfrak{C}\cdot d_2(x)\big)\, .
\end{equation}
Of course, $k$ also depends on $x$, as do $\mathfrak{A}$, $\mathfrak{B}$, $\mathfrak{C}$
and $\mathfrak{D}$.
So from all of this we obtain
\begin{equation}
\label{eq:A2formula2d}
	\mathcal{A}_2(x,w)\, =\, \frac{1}{\pi\, \sqrt{\Xi(x)}}\, \int_{-1}^{1} 
\frac{\mathcal{Q}_1(x;\Phi^{-1}(s))}{\mathcal{Q}_2(x,w;\Phi^{-1}(s))}
\cdot \frac{1}{\sqrt{(1-s^2)(1-k^2s^2)}}\, ds\, ,
\end{equation}
assuming we have found the linear fractional transformation $\Phi$ such that (\ref{eq:PhiCondition})
holds.
\begin{lemma} For $x,w \in (0,\infty)$ satisfying $4x+w^2<1$, we have
\begin{equation}
	\mathcal{A}_2(x,w)\, =\, \frac{2}{\pi\, \sqrt{1-4x}}\, \int_{-(1+4x)^{-1/2}}^{-(1-4x)^{1/2}} \frac{1}{\sqrt{\big(1-(1+4x)r^2\big)\big(1-(1-4x)^{-1}r^2\big)}}
\cdot \frac{\mathcal{Q}_1\big(x;\mathcal{L}^{-1}(r)\big)}{\mathcal{Q}_2\big(x,w;\mathcal{L}^{-1}(r)\big)}\, dr\, .
\end{equation}
Changing variables, by letting $t=r/\sqrt{1-4x}$, this gives
\begin{equation}
	\mathcal{A}_2(x,w)\, =\, \frac{2}{\pi}\, \int_{-(1-16x^2)^{-1/2}}^{-1} \frac{1}{\sqrt{(1-t^2)\big(1-(1-16x^2)t^2\big)}}
\cdot \frac{\mathcal{Q}_1\big(x;\mathcal{L}^{-1}(r\, \sqrt{1-4x})\big)}{\mathcal{Q}_2\big(x,w;\mathcal{L}^{-1}(r\, \sqrt{1-4x})\big)}\, dr\, .
\end{equation}
\end{lemma}
This still leaves the problem of determining the transformation behavior for $\mathcal{Q}_1(x;\cdot)/\mathcal{Q}_2(x,w;\cdot)$. Fortunately, one easy solution is to obtain the partial fraction decomposition for this rational function.

\begin{lemma}
Suppose $\mathfrak{A}\mathfrak{D}-\mathfrak{B}\mathfrak{C}\neq 0$ and suppose we have 
$\rho \in \C \cup \{\infty\} \setminus \{-\mathfrak{D}/\mathfrak{C}\}$,
then for $\Phi$ and $\Phi^{-1}$ as in (\ref{eq:PhiDef}) and (\ref{eq:PhiInvDef}), we have
\begin{equation}
\label{eq:PFtransform}
	\frac{1}{\Phi^{-1}(\xi)-\rho}\, =\, -\frac{\mathfrak{C}}{\mathfrak{C}\rho +\mathfrak{D}} 
+ \frac{\mathfrak{A}\mathfrak{D}-\mathfrak{B}\mathfrak{C}}{(\mathfrak{C}\rho+\mathfrak{D})^2} \cdot \frac{1}{\xi - \Phi(\rho)}\, .
\end{equation}
\end{lemma}
\begin{proof}
Simply calculate
\begin{equation}
	\Phi^{-1}(\xi) - \rho\, =\, \frac{\mathfrak{D}\xi-\mathfrak{B}}{-\mathfrak{C}\xi+\mathfrak{A}} - \rho\,
	=\, \frac{\mathfrak{D}\xi-\mathfrak{B} +\mathfrak{C}\xi\rho-\mathfrak{A}\rho }{-\mathfrak{C}\xi+\mathfrak{A}}\,
	=\, \frac{\left(\mathfrak{D}+\mathfrak{C}\rho\right)\xi-\left(\mathfrak{A}\rho +\mathfrak{B}\right)}{-\mathfrak{C}\xi+\mathfrak{A}} \, .
\end{equation}
So
\begin{equation}
	\frac{1}{\Phi^{-1}(\xi) - \rho}\, 
	=\, \frac{-\mathfrak{C}\xi+\mathfrak{A}}{\left(\mathfrak{D}+\mathfrak{C}\rho\right)\xi-\left(\mathfrak{A}\rho +\mathfrak{B}\right)}\,
	=\, \frac{1}{\mathfrak{C}\rho+\mathfrak{D}}
\cdot \frac{-\mathfrak{C}\xi+\mathfrak{A}}{\xi-\frac{\mathfrak{A}\rho +\mathfrak{B}}{\mathfrak{C}\rho+\mathfrak{D}}} \, .
\end{equation}
Rewriting this as a partial fraction,
\begin{equation}
	\frac{1}{\Phi^{-1}(\xi) - \rho}\, 
	=\, -\frac{\mathfrak{C}}{\mathfrak{C}\rho+\mathfrak{D}}+\frac{1}{\mathfrak{C}\rho+\mathfrak{D}}
\cdot \frac{\mathfrak{C}\xi-\mathfrak{C}\cdot \frac{\mathfrak{A}\rho +\mathfrak{B}}{\mathfrak{C}\rho+\mathfrak{D}}-\mathfrak{C}\xi+\mathfrak{A}}{\xi-\frac{\mathfrak{A}\rho +\mathfrak{B}}{\mathfrak{C}\rho+\mathfrak{D}}} \, .
\end{equation}
Simplifying, this is equal to (\ref{eq:PFtransform}).
\end{proof}
We mention a useful reference for simplifying elliptic integrals \cite{DLMF}, more specifically Section 19.14
\begin{equation*}
	\text{\url{https://dlmf.nist.gov/19.14}.}
\end{equation*}
One reference therein is a well-known preprint \cite{LabahnMutrie} available from the web link
\begin{equation*}
	\text{\url{https://cs.uwaterloo.ca/~glabahn/Papers/ellipticPrePrint.pdf}.}
\end{equation*}
Let us now give the partial fraction decomposition of  $\mathcal{Q}_1(x;\cdot)/\mathcal{Q}_2(x,w;\cdot)$,
and also determine the linear fractional transformation $\Phi$ that 
satisfies (\ref{eq:PhiCondition}).

\begin{lemma} Let $\mathcal{Z}(x,w) = \{\mathcal{A}_1(x,w),\mathcal{A}_2(x,w),b_1(x,w),b_2(x,w)\}$, which is the variety
of roots of $\mathcal{Q}_2(x,w;\cdot)$. Then,
as meromorphic function on $\C \cup \{\infty\}$,
we have the partial fraction decomposition of $\mathcal{Q}_1(x;\cdot)/\mathcal{Q}_2(x,w;\cdot)$:
\begin{equation}
\label{eq:FirstPartialFraction}
\forall z \in \C \cup \{\infty\} \setminus \mathcal{Z}(x,w)\, ,\qquad	\frac{\mathcal{Q}_1(x;z)}{\mathcal{Q}_2(x,w;z)}\,
	=\, 1 + \sum_{\rho \in \mathcal{Z}(x,w)} \frac{1}{z-\rho} \cdot \frac{w^2 \rho^2}{x^2} 
\prod_{\sigma \in \mathcal{Z}(x,w) \setminus \{\rho\}} \frac{1}{\rho-\sigma}\, .
\end{equation}
\end{lemma}
\begin{proof}
Note that by equation (\ref{eq:calQ2Def}) we have
\begin{equation}
\forall \rho \in \mathcal{Z}(x,w)\, ,\qquad
	\mathcal{Q}_1(x;\rho)\, =\, \mathcal{Q}_2(x,w;\rho) + w^2 \rho^2\, =\, w^2 \rho^2\, ,
\end{equation}
because $\mathcal{Z}(x,w)$ is the variety of roots of $\mathcal{Q}_2(x,w;\cdot)$.
By equation (\ref{eq:Q2factorization}), we know that for each $\rho \in \mathcal{Z}(x,w)$,
we have
\begin{equation}
	\frac{1}{\mathcal{Q}_2(x,w;z)}\, =\, \frac{1}{x^2 \cdot (z-\rho)} \prod_{\sigma \in \mathcal{Z}(x,w) \setminus \{\rho\}} \frac{1}{z-\sigma}\, .
\end{equation}
Therefore, we have
\begin{equation}
\forall \rho \in \mathcal{Z}(x,w)\, ,\qquad
	\lim_{z \to \rho} \frac{\mathcal{Q}_1(x;\rho)}{\mathcal{Q}_2(x,w;z)}\,
	=\, \frac{w^2 \rho^2}{x^2}\, \prod_{\sigma \in \mathcal{Z}(x,w) \setminus \{\rho\}} \frac{1}{\rho-\sigma}\, .
\end{equation}
That gives the coefficient of $1/(z-\rho)$ for each of the roots. For the limit as $z \to \infty$, we have 
\begin{equation}
\lim_{z \to \infty} \frac{\mathcal{Q}_1(x;z)}{\mathcal{Q}_2(x,w;z)}\, =\, 1\, ,
\end{equation} 
because the leading coefficient of each
of the quartics $\mathcal{Q}_1(x;\cdot)$ and $\mathcal{Q}_2(x,w;\cdot)$ is the same, $x^2$.
Therefore, the meromorphic function equal to the left-hand-side minus the right-hand-side
of (\ref{eq:FirstPartialFraction}) is a function with no poles and which converges to $0$ at $\infty$.
So, by Liouville's theorem from complex analysis, this difference is identically $0$.
In other words, equation (\ref{eq:FirstPartialFraction}) is valid.
\end{proof}

%Now let us specify the linear fractional transformation.
%\begin{lemma} Define the linear fractional transformation
%\begin{equation}
%\Phi^{-1}(x;z)\, =\, \frac{1}{(1-4x^2)^{1/4}} \cdot \frac{1 + z \cdot  (1-4x^2)^{1/4}}{1- z\cdot (1-4x^2)^{1/4}}\, .
%\end{equation}
%and
%\begin{equation}
%\Phi(x;\xi)\, =\, \frac{1}{(1-4x^2)^{1/4}} \cdot \frac{1 + z \cdot  (1-4x^2)^{1/4}}{1- z\cdot (1-4x^2)^{1/4}}\, .
%\end{equation}
%\end{lemma}

\baselineskip=12pt
\bibliographystyle{plain}

\end{document}